\documentclass{arxiv}
\RequirePackage{tgtermes}
\RequirePackage{newtxtext}
\RequirePackage{newtxmath}
\RequirePackage{bm}
\RequirePackage{endnotes}

\OneAndAHalfSpacedXII % Current default line spacing
\usepackage[dvipsnames]{xcolor}
\usepackage[colorlinks,urlcolor=blue,linkcolor=blue,citecolor=blue]{hyperref}
\usepackage{framed}
\usepackage{wrapfig}% to insert images surrounded by text
\usepackage{enumitem}
\usepackage{float}
\usepackage{url}

\usepackage{booktabs}
\usepackage{tabularx}
\usepackage{array}
\newcolumntype{L}{>{\bfseries\raggedright\arraybackslash}p{0.19\textwidth}}
\newcolumntype{Y}{>{\raggedright\arraybackslash}X}

\usepackage{amsmath}
\usepackage{amssymb}

\usepackage[caption=false]{subfig}

\usepackage[capitalise]{cleveref}

\usepackage[sort&compress]{natbib}
 \bibpunct[, ]{[}{]}{,}{n}{}{,}%
 \def\bibfont{\small}%
\newcommand{\defn}{\stackrel{\triangle}{=}}
\newcommand{\tth}{{}^{\text{th}}}

\newcommand{\fbar}{\overline{f}}

\newcommand{\tq}{\widetilde{q}}
\newcommand{\tvq}{\widetilde{\vq}}

\newcommand{\bE}{\mathbb{E}}
\newcommand{\bR}{\mathbb{R}}
\newcommand{\bZ}{\mathbb{Z}}

\newcommand{\cT}{\mathcal{T}}
\newcommand{\cX}{\mathcal{X}}
\newcommand{\cZ}{\mathcal{Z}}

\newcommand{\ve}{\boldsymbol{e}}
\newcommand{\vq}{\boldsymbol{q}}
\newcommand{\vone}{\boldsymbol{1}}

\newcommand{\vlambda}{\boldsymbol{\lambda}}
\newcommand{\vmu}{\boldsymbol{\mu}}

\newcommand{\E}[1]{\bE\left[#1\right]}
\newcommand{\ind}[1]{\mathbb{I}_{\left\{#1 \right\}}}
\newcommand{\Prob}[1]{\mathbb{P}\left[#1\right]}

\EquationsNumberedThrough    % Default: (1), (2), ...
\TheoremsNumberedThrough     % Preferred (Theorem 1, Lemma 1, Theorem 2)
\ECRepeatTheorems  %  

\MANUSCRIPTNO{}

\begin{document}
%%%%%%%%%%%%%%%%

%\setcounter{page}{1}
%\DOI{10.1287/xxxx.0000.0000}%

\RUNAUTHOR{Hurtado-Lange and Grosof}

\RUNTITLE{JSQ with Markov-Modulated parameters}

\TITLE{Markov Modulated JSQ in Heavy Traffic Via the Poisson Equation}

\ARTICLEAUTHORS{%
%\AUTHOR{John Doe,\textsuperscript{a} Jane Smith,\textsuperscript{b}}
%\AFF{\textsuperscript{a}Department of Industrial Engineering, University of XYZ, \EMAIL{john.doe@xyz.edu; \textsuperscript{b}Department of Computer Science, University of ABC, \EMAIL{jane.smith@abc.edu}} 
\AUTHOR{Daniela Hurtado-Lange}
\AFF{Operations Department, Kellogg School of Management,
Northwestern University, \EMAIL{daniela.hurtado@kellogg.northwestern.edu}}

\AUTHOR{Izzy Grosof}
\AFF{Industrial Engineering \& Management Sciences Department,
Northwestern University, \EMAIL{izzy.grosof@northwestern.edu}}
} % end of the block

\ABSTRACT{%
% Enter your abstract
In parallel-server systems with a single stream of arrivals (a.k.a. load balancing), Join-the-Shortest-Queue (JSQ) is a popular routing algorithm. There is extensive literature studying this system in various asymptotic regimes, but all assume constant parameters (arrival and service rates). 
We study the JSQ system with Markov-modulated parameters and heterogeneous servers. Our main contributions are: (i) We compute the heavy-traffic distribution of the scaled vector of queue lengths; (ii) We utilize a novel hybrid methodology that combines the Transform Method for queue-lengths analysis and the Poisson equation; (iii) We provide sufficient conditions to ensure state space collapse, showing provable balancing power of JSQ for heterogeneous servers. Unlike other studies involving Markov-modulated queues, these conditions don't depend on the mixing time of the modulating chain and are valid for a countably infinite state space.
We numerically demonstrate the strength of our results under moderate traffic intensities and showcase their independence from the corresponding mixing times. 
}%

%\FUNDING{}

%Supplemental Material:
%Data Ethics & Reproducibility Note:

%\KEYWORDS{Stochastic programming, Decision support,Uncertainty, Disaster response, Optimization}

%\HISTORY{Received: Month DD, YYYY; Accepted: Month DD, YYYY; Published Online: Month DD, YYYY}

\maketitle
%%%%%%%%%%%%%%%%%%%%%%%%%%%%%%%%%%%%%%%%%%%%%%%%%%%%%%%%%%%%%%%%%%%%%%

% Text of your paper here

\section{Introduction}

Parallel-server systems with a single stream of arrivals that are routed upon joining have been studied for decades. Since the classic papers \cite{Fos-Sal-1978-jsq,Win-1977-JSQ-optimality} showed that joining the shortest queue (JSQ) minimizes the waiting times across policies that don't know job durations, this routing policy has received a lot of attention in the literature. Most of the studies analyze this JSQ system in an asymptotic regime, such as classical heavy traffic (e.g. \cite{Ery-Sri-2012-drift,Fos-Sal-1978-jsq,HL-Mag-2020-MGF,Win-1977-JSQ-optimality}) and many-server heavy-traffic regimes where the number of servers grows at various rates with respect to the heavy-traffic parameter (e.g. \cite{HL-Mag-2022-JSQ-Many-Server,Muk-2022-ROC-JSQ,Jhu-HL-Mag-2023-JSQ-ROC}). More recently, there has been some interest in analyzing pre-limit systems and their rate of convergence to the corresponding regimes \cite{Ban-Muk-2020-JSQ-HW-sensitivity,Jhu-HL-Mag-2023-JSQ-ROC}. All these studies complement each other to provide a rich understanding of the JSQ policy across systems of varying sizes and loads. However, they are all constrained to settings where the arrival and service rates are constant over time.

One of the most intuitive applications for the JSQ system is the checkout area in supermarkets. After customers have collected the items they want to purchase, they observe the cashiers' area, and will likely join the shortest queue and wait there until they are first in line and get processed. Clearly, the customer arrival rate is not constant throughout the day. On the servers' side, because the cashiers are typically human, they do not process all customers at the same speed. Further, they frequently have baggers that accelerate their service rate, but the baggers are not always helping the same server. 

Situations in which the servers have two possible service rates are frequently modeled as working vacations: the server works at full speed or slows down without completely shutting down. There is plenty of literature on queues where servers take working vacations, but most of it focuses on a single queue. See, for example, \citet{Fie-2025-working-vacation-survey} and the references therein. Further, having only two possible server speeds is also limiting. In the supermarket example, this model captures the situation with a traveling bagger but does not allow each cashier to have different service rates during the day. For example, they may get tired during the shift and slow down towards the end.

These time-varying arrival and service rates also arise across diverse application domains—from cloud and datacenter job routing, to call centers with time-dependent demand, to ride-sharing platforms where demand surges generate sharp fluctuations in system load.

Queueing models with constant arrival and service rates typically have i.i.d. inter-arrival and service times. This i.i.d. structure enables deep analytical results, which makes it convenient from a mathematical perspective. However, it is unclear to what extent these assumptions limit the practical implications of the results and properties. A much more general model for the arrival and service rates is Markov modulated, that is, the rates are time-varying, and their transitions follow a Markov chain. Indeed, Markov-modulated arrival processes are dense in the space of point processes \cite[Chapter XI]{Asm-2003-Applied-Prob-book}. Therefore, this model encompasses a wide range of arrivals and service processes, while maintaining the tractable structure of a Markov chain. 

The main barrier to analyzing Markov-modulated JSQ is that the queue lengths, arrival rates, and service rates are interdependent, destroying the independence structure that underpins classical queueing-theory analysis. Existing approaches for Markov-modulated single-server queues rely heavily on matrix-analytic methods, finite-state modulation, or diffusion arguments that do not scale to multi-server routing. Moreover, existing heavy-traffic methods typically require mixing-time assumptions and yield convoluted equations. In this paper, we study a JSQ system with Markov-modulated arrival and service rates in heavy traffic. Our methodology circumvents these obstacles by isolating modulation dependence via the Poisson equation, thereby extending the transform method for heavy-traffic analysis without introducing additional technical machinery (see \cref{sec:proof-overview} for an overview of these approaches). Specifically, our contributions are
\begin{enumerate}
    \item We compute the heavy-traffic distribution of the scaled vector of queue lengths and the rate of convergence of the component-wise Moment Generating Function (MGF). To the best of our knowledge, this is the first paper to analyze a multi-queue system where both arrival and service rates are not i.i.d. (see \cref{thm:jsq-vq-convergence}).
    
    \item We develop a novel technique that combines the transform method introduced by \citet{HL-Mag-2020-MGF} with the Poisson equation to compute the MGF of the scaled queue lengths. 
    
    \item Specifically, we use the Poisson equation to compute the covariance between functions of the modulating Markov chain and the queue lengths (see \cref{thm:ssq-poisson-lemma,thm:jsq-poisson-lemma}). To the best of our knowledge, this is the first paper to use the Poisson equation to compute a covariance in the queueing literature.
      
    \item Our analysis technique is the first to allow infinite modulation processes. As discussed in \cref{sec:prior-work}, there is extensive literature in single-server queues with Markov-modulated parameters, but they all assume a finite-state-space modulation process.
    
    \item We establish a sufficient condition for JSQ to balance the queue lengths in heavy traffic. We call it the State Space Collapse (SSC) load condition (see \cref{def:condition_lambda_i-mu_i}). This condition depends only on the arrival and service rates in each state of the Markov-modulated process.

    \item Our methodology replaces mixing-time assumptions with explicit, verifiable conditions on the Markov-modulated arrival and service rates.
\end{enumerate}

Our approach is particularly innovative in the Markov-modulated queueing literature, as it does not require bounding the mixing time of the Markov-modulation process. Instead, we show that a structural SSC load condition is sufficient, even with infinitely many (but countable) modulating states. This is the first heavy-traffic method capable of handling infinite-state modulation in load-balancing systems.  

The rest of this paper is organized as follows. \cref{sec:prior-work} describes relevant literature for our work. \cref{sec:model} describes the JSQ model in detail, establishes the notation, and briefly discusses the main results of the paper. \cref{sec:proof-overview} is essential to the paper as it describes our proof technique, with special attention to the transform method introduced by \citet{HL-Mag-2020-MGF} and our integration of the Poisson equation. \cref{sec:ssq} introduces the details of our methodology in the context of a single-server queue. This simplified setting allows us to keep the exposition clean and emphasize the core steps of our proof technique. \cref{sec:jsq} generalizes the proof to the JSQ system. In \cref{sec:empirical} we provide numerical evidence that our results are independent of mixing times, and show that the SSC load condition is not a necessary condition. Finally, \cref{sec:conclusion} discusses the key takeaways of the paper and some future work directions.

\section{Prior Work}
\label{sec:prior-work}

Our paper lies at the intersection of three lines of literature (two of them closely related): (i) Load balancing systems and Join-the-Shortest Queue (JSQ); (ii) Heavy-traffic analysis; and (iii) Analysis of Markov-modulated queueing systems. We discuss the literature and our contributions to each in the following subsections \cref{sec:prior-jsq,sec:prior-ht,sec:prior-markov}, as well as broader models in \cref{sec:prior-varying}.

\subsection{Load balancing and JSQ}
\label{sec:prior-jsq}
% About JSQ \\
% - Queueing networks are intractable, so heavy traffic analysis
% - JSQ is a great example to start a new theory as it carries multiple difficulties but is intuitive. Hence, it is a very important system in the queueing literature (can mention guardrails paper by Izzy)
% - JSQ is heavy-traffic optimal
% - JSQ has also been studied in other asymptotic regimes such as many-server heavy-traffic at different loads (see my many-server paper for references)

JSQ has been studied for decades, with the first analyses showing that it is response-time optimal across policies that do not know jobs durations \cite{Win-1977-JSQ-optimality,Fos-Sal-1978-jsq, Web-1978-JSQ-optimal, Eph-1980-JSQ}, in a time-homogenous exponential setting. One of the most striking features of JSQ is that it balances the queue lengths in heavy traffic, that is, all the queues are approximately equal in the limit. Formally, a JSQ system satisfies 
State Space Collapse (SSC) to a one-dimensional subspace. Intuitively, this result holds because whenever there is an arrival, it goes to the queue with the fewest jobs, decreasing the difference between the longest and shortest queues. Hence, as the traffic intensity increases, there are sufficient arrivals to balance all the queue lengths. 
% References from my many-server paper:
% Heavy traffic analysis: 10, 12, 15, 17, 19, 25, 26, 27, 29, 30, 37, 41, 47, 57, 58

The importance of JSQ systems goes beyond heavy traffic analysis. It has been studied in many-server settings \cite{Mit-1996-JSQ-d,Vve-Dob-Kar-1996-JSQ2} and many-server heavy-traffic at various traffic intensities \cite{Ban-Muk-2019-JSQ-HW,Ban-Muk-2020-JSQ-HW-sensitivity,Bra-2020-JSQ,Esc-Gam-2018-JSQ,Jhu-HL-Mag-2023-JSQ-ROC}. In all these studies, the servers process jobs in order of arrival after routing. However, delay can be considerably reduced if a routing algorithm is paired with improved scheduling at the servers' level. \citet{Gro-Scu-Har-2019-guardrails}, for example, uses JSQ as a benchmark to show the improvement in response time that can be achieved by other routing algorithms when paired with optimal SRPT scheduling, and JSQ has also been studied alongside processor-sharing scheduling \cite{Gup-2007-Analysis}.

In this paper, we use JSQ as a foundational system to develop and showcase our methodology for queueing systems with Markov-modulated arrival and service rates.

\subsection{Heavy-traffic analysis and transform method}
\label{sec:prior-ht}
% How to do heavy traffic analysis?
% - Diffusion limits for decades \\ 
% - Three popular direct methods: BAR (Anton's literature + multiscale heavy-traffic), Stein's method (Anton + Lei + Daniela many-server paper), MGFs (Sem Borst and Daniela's transform method stuff discussed below) \\
% \textbf{Transform method literature:} \\
% - Daniela's MGF method papers \\
% - Sushil's transform method for two-sided queues \\
% - Prakirt's and Daniela's ROC paper \\
% - Zaiwei: Transform method for RL\\
Exact analysis of queueing systems is typically intractable. Even for the simplest system (the $M/M/1$ queue), the distribution of queue lengths is known in closed form only in steady state. The transient distribution of the queue length requires numerical approximations for Bessel functions and double-Laplace transforms, suggesting that exact analysis of more complex systems is indeed intractable. See, for example, \cite[Chapter III.8]{Asm-2003-Applied-Prob-book}. Hence, a prevalent approach is to study asymptotic regimes to draw key insights. Heavy traffic and mean field (or many-servers) are two of the most popular regimes, and JSQ systems have been pivotal in developing new methodologies and proving new results.

For decades, the most popular tool for heavy-traffic analysis was diffusion approximations and process-level convergence. This approach was established by \citet{Kin_1962_RBM}, where a $G/G/1$ queue is analyzed in heavy traffic, and has been generalized to multiple settings. A non-exhaustive list of foundational papers and books on this topic are \cite{Bell-Will-01-Nsystem,Williams-2000-CRP,Har-1988-Brownian,HarLop-1999-CRP,Har_2013_book}, but the literature is extensive. In simple words, under this approach, one scales the queue lengths and shows process-level convergence to a Reflected Brownian Motion (RBM) and then analyzes the steady-state behavior of this RBM. The analysis is only complete if the interchange of limits property can be proved. This last step is tremendously challenging and is frequently done in a separate paper.

In the last decade, the queueing community has developed direct methods that do not require this interchange of limits step. Instead, they analyze the heavy-traffic limit of the steady-state scaled queue lengths directly. In very simple words, Stein's method \cite{Bra-Dai-2017-Stein,Bra-Dai-2017-Stein2,Wal_2020_SteinHT,Liu-Yin-2020-Stein-SubHW,Ying-2017-Stein,HL-Mag-2022-JSQ-Many-Server} uses the generator of a Markov process to bound the Wasserstein's distance between the scaled queue lengths and their limiting distribution; the Basic Adjoint Relationship (BAR) methodology bounds the jumps on exponential test functions of the workload process to obtain the heavy-traffic distribution \cite{Bra-Dai-Miy_2017_BAR,Bra-Dai-Miy-2025-BAR,Dai-Huo-2024-Multi-Scale-HT,Dai-Gly-Yao-2023-Multi-Scale-HT}; and the drift and transform methods bound the drift of a test function that yields information about the moments of the queue lengths in heavy traffic \cite{Ery-Sri-2012-drift,Mag-Sri-2016-switch,Wang-Mag-Sri-2022-Bandwidth-Sharing,HL-Mag-2020-MGF,HL-Mag-2022-JSQ-Many-Server,Jhu-HL-Mag-2023-JSQ-ROC}. The literature presented in this paragraph is by no means extensive for each methodology and instead, represents a small sample of the current work on each of them. 

In this paper, we focus on the transform method, which was first developed for the JSQ system and a parallel-server system with control on the service process by \citet{HL-Mag-2020-MGF}. In more recent years, it has been generalized to multiple settings, such as a JSQ system with abandonment by \citet{Jhu-Zub-Mag_2025_JSQ-abandonments}, other parallel-server systems with routing by \citet{Luo-Zub-2025-Load-Balancing-Stability}, ride-sharing settings by \citet{Var-Mag-2022-HT-two-sided-q}, and stochastic gradient descent algorithms by \citet{Che-Mou-Mag-2022-SGD-Transform}. Each extension follows the same drift-based template with an exponential test function, but requires additional ideas to address system-specific challenges.

% Everything above is for fixed parameters! \\
All the literature described above assumes fixed arrival and service rates. In this paper, we enter the much less understood domain of queueing systems with time-varying parameters. In particular, we develop a methodology that uses the transform method and the Poisson equation to study queueing systems with Markov-modulated parameters.

\subsection{Queues with Markov-modulated parameters}
\label{sec:prior-markov}

% - Very popular in 80s and 90s \\
% - These methods are not easy to generalize and depend on matrix properties \\
%%%% Burman, Smith (1986): Markov/M/1 queue in light and heavy traffic. Approximation of mean.
%%%% Prabhu, Zhu (1989): Busy period and mean workload in a Markov/Markov/1 queue
%%%% Baccelli, Makowski (1991): Markov/G/1 queue analysis using Martingales
%%%% Falin, Falin (1999): Markov/G/1 queue-length distribution in heavy traffic
%%%% Mahabhashyam, Gautam (2005): Single-server queue with Markov-modulated service times

% \cite{Tho-Ver-2016-Markov-G-1-processor-sharing} Processor sharing and Markov-modulated single queue. Heavy-traffic distribution of scaled workload
% Siva and Shangcong's paper: Heavy traffic analysis for multiple queues, but service times are deterministic. Also, they use multi-step Lyapunov approach and obtain results in terms of auto-covariance function. They also need assumptions on the mixing time.

% - These methods are not easy to generalize and depend on matrix properties \\
% Nobody has done an analysis of routing policies in heavy traffic. In particular, we provide conditions under which JSQ satisfies the CRP condition and, hence, is heavy-traffic optimal.

Much of the literature analyzing queueing systems with Markov-modulated parameters dates to the 1980s, 1990s, and early 2000s. One of the seminal papers studying performance metrics of a single queue is \citet{Bur-Smi_1986_Markov-G-1-HT}, where the authors compute the mean queue length and mean delay in the light and heavy-traffic limits, subject to a technical assumption relating to an interchange of limits. The literature is extensive, with several prominent papers, including \cite{Pra-Zhu-1989-Markov-q,Bac-Mak-1991-M-G-1-martingale,Fal-Fal-1999-Markov-G-1-HT}. This is, of course, not an extensive list of papers. The focus in most of these papers is on managing the transition matrix of the Markov-modulating process and obtaining a tractable expression for diverse performance metrics such as the busy periods, queue length, and delay. For example, \citet{Bac-Mak-1991-M-G-1-martingale} propose a methodology based on martingale analysis. This body of literature analyzes single-server queues, and its contributions are oriented toward matrix-analytic methods and approaches that support a general distribution of service times. That said, all of these constrain the Markov-modulation process to a finite state space.

More recent studies on queueing systems with Markov-modulated parameters are scarce. \citet{Blo-et-al-2014-Markov-G-infty}, for example, use techniques inspired by the Central-Limit-Theorem to analyze an $M/G/\infty$ queue with Markov-modulated arrival rate. The concept of working vacations also gained popularity recently, with \citet{Fie-2025-working-vacation-survey} presenting a complete survey of recent studies. Again, these literature streams focus on a single queue, and the Markov-modulation process is finite.

The most relevant studies to our paper are \citet{Fal-Fal-1999-Markov-G-1-HT} and \citet{Mou-Mag-2024-Switch-Markov}. \citet{Fal-Fal-1999-Markov-G-1-HT} study a single-server queue with Markov-modulated arrivals and general service times. They use matrix-analytic methods to compute the heavy-traffic distribution of queue lengths via Laplace transforms and a Poisson equation. However, their analysis of the Markov-modulation process only allows for finite state space, and their matrix analysis methods are targeted to this specific system. Hence, generalizing this method to contexts beyond the single-server setting is challenging and, indeed, it has not been done. Their result is also independent of the mixing time of the modulating chain, but they do not incorporate SSC in their analysis. The only known generalization of this paper is \citet{Dim-2011-Falin-Generalization}, which extends the methodology to allow Markov-modulated service times as well, still in the single-server setting.

On the other hand, \citet{Mou-Mag-2024-Switch-Markov} compute the heavy-traffic queue-length distribution of an input-queued switch using an extension of the transform method. The similarities between this paper and ours are: (i) they also study a system with multiple queues and multiple servers; (ii) their proof also uses SSC; and (iii) they also utilize transform-method techniques to compute the distribution of the queue lengths. However, the input-queued switch is best analyzed in discrete time, as it has fixed service times of one time slot per job, and SSC to a one-dimensional subspace is obtained only under a carefully chosen heavy-traffic regime in which the arrival rates are tuned so that exactly one of the $n^2$ queues is critically loaded, while all others remain strictly stable. Additionally, they study $M$-step drift of the test function and let $M\to\infty$ to compute the distribution. Hence, their technique depends on the mixing time of the Markov-modulating process, and they provide the result in terms of the auto-covariance function. The key distinctions between this paper and ours are (i) our results are not dependent on the mixing time; (ii) our methodology allows for Markov-modulated arrival and service rates; (iii) we provide sufficient conditions for SSC depending on the parameters; and (iv) our approach, based on the transform method and the Poisson equation, yields an elegant proof that doesn't require taking additional limits. A summary of the key differences between \cite{Fal-Fal-1999-Markov-G-1-HT,Dim-2011-Falin-Generalization,Mou-Mag-2024-Switch-Markov} and our paper is presented in \cref{tab:lit-review}. 
% \begin{table}
% \TABLE
% % Caption
% {Comparison of relevant literature and our paper.\label{tab:lit-review}}
% % Table
% {
% \begin{tabular}{|l|p{0.22\textwidth}|p{0.22\textwidth}|p{0.22\textwidth}|}
%         \hline 
%         Paper & \cite{Fal-Fal-1999-Markov-G-1-HT,Dim-2011-Falin-Generalization} & \cite{Mou-Mag-2024-Switch-Markov} & This paper \\ \hline
%         Model & Single-server queue & Input-queued switch & Parallel servers, JSQ \\ \hline
%         Modulation & Arrivals and service & Arrivals & Arrivals and service \\ \hline
%         Modulation state space & Finite & Finite & Countable \\ \hline
%         \# servers and queues & 1 & $n$ servers, $n^2$ queues & $n$ servers and queues \\ \hline
%         State Space Collapse? & No & Assumed & Proved \\ \hline
%         Mixing-time dependence & Yes & Yes & No \\ \hline
%         Method & Transform, Poisson equation and matrix analytics & $M$-step drift & Transform and Poisson equation \\ \hline
%     \end{tabular}}
% % Notes
% {}
% \end{table}
\begin{table}
\TABLE
% Caption
{Comparison of relevant literature and our paper.\label{tab:lit-review}}
% Table
{
\begin{tabularx}{\textwidth}{@{} l Y Y Y @{}}
        \toprule 
        Reference & \citet{Fal-Fal-1999-Markov-G-1-HT}, 
        \citet{Dim-2011-Falin-Generalization} & \citet{Mou-Mag-2024-Switch-Markov} & This paper \\ \specialrule{0.4pt}{0pt}{1pt}
        Model & Single-server queue & Input-queued switch & Parallel servers, JSQ \\ \specialrule{0.4pt}{0pt}{1pt}
        Modulation & Arrivals and service & Arrivals & Arrivals and service \\ \specialrule{0.4pt}{0pt}{1pt}
        Modulation space & Finite & Finite & Countable \\ \specialrule{0.4pt}{0pt}{1pt}
        \# servers and queues & 1 & $n$ servers, $n^2$ queues & $n$ servers and queues \\ \specialrule{0.4pt}{0pt}{1pt}
        SSC? & No & Proved for a specialized heavy-traffic regime & Proved  \\ \specialrule{0.4pt}{0pt}{1pt}
        Mixing-time dependence & No & Yes & No \\ \specialrule{0.4pt}{0pt}{1pt}
        Method & Transform + Poisson eq. + matrix analytics & Transform + $M$-step drift & Transform + Poisson eq. \\ \bottomrule
    \end{tabularx}}
% Notes
{}
\end{table}

To the best of our knowledge, no existing work analyzes a multi-server routing system with both arrival and service rates modulated by a (possibly infinite state space) Markov chain, derives SSC conditions, and obtains explicit heavy-traffic queue-length distributions. Even the simplest single-server works require finite-state-space modulation and substantial matrix analysis.

\subsection{Other time-varying models}
\label{sec:prior-varying}

In addition to Markov-modulated queues, queues with general time-varying dynamics have also been studied. \citet{Whitt-2018-Varying} gives a broad overview of the field. \citet{Bambos-2004-Queueing} study a general stochastic time-inhomogenous multiclass scheduling model, and demonstrate that a family of MaxWeight analogues achieve stability optimality. These models are much broader than our Markov-modulated model, but generally too complex for the performance analysis results, such as distribution results on queue length, that we focus on here.

\section{Model and Main Results}
\label{sec:model}

Consider an ergodic Continuous-Time Markov Chain (CTMC) $\{Z(t):t\in\bR_+\}$ with countable state space $\cZ$, transition rate matrix $Q=[\alpha_{ii'}]_{i,i'\in\cZ}$ and stationary distribution $\pi$. We use $\alpha_{i\bullet}$ to denote the rate at which the Markov chain leaves state $i\in\cZ$, that is,
\begin{align*}
    \alpha_{i\bullet}\defn \sum_{i'\in\cZ,\,i'\neq i} \alpha_{ii'}.
\end{align*}
Note that $\alpha_{ii}=-\alpha_{i\bullet}$ by definition of the transition rate matrix. 

We refer to the CTMC $\{Z(t):t\in \bR_+\}$ as the \textit{modulating Markov chain}, as its states control the arrival and service rate of the queueing system we study.

We consider a load-balancing system operating in continuous time where arriving jobs are routed according to Join-the-Shortest-Queue (JSQ), that is, arriving jobs are routed to the server with the least number of jobs (breaking ties at random). There are $n$ heterogeneous servers, each with an infinite buffer. Once jobs are assigned to one of them, they wait in the corresponding line until processed and they are not allowed to jump to another server's line. Let $\vq(t)$ be the vector of queue lengths at time $t$, where $q_j(t)$ represents the number of jobs at server $j$, including any in service.

The inter-arrival and service times are exponentially distributed, and the rates depend on the state of the modulating Markov chain $\{Z(t):t\in\bR_+\}$. Specifically, when $Z(t)=i$, the arrival rate to the system is $\lambda_i$, and the service rate of server $j$ is $\mu_{ij}$, for each $j\in\{1,\ldots,n\}$. Then, the state of the system is the tuple $(Z(t),\vq(t))$.

We assume there exist constants $\lambda_{\max}<\infty$ and $\mu_{\max}<\infty$ such that $\lambda_i<\lambda_{\max}$ and $\mu_{i,j}<\mu_{\max}$ for all $i\in\cZ$ and all $j\in\{1,\ldots,n\}$. We define the total service rate in state $Z(t)=i$ by
\begin{align*}
    \mu_{i,\Sigma} \defn \sum_{j=1}^n \mu_{i,j},
\end{align*}
the long-run average arrival rate by $\lambda \defn \E{\lambda_i\,|\, i\sim \pi}$, service rate for server $j$ by $\mu_j \defn \E{\mu_{i,j}\,|\, i\sim \pi}$, and total service rate by $\mu_\Sigma\defn \sum_{j=1}^n \mu_j$. 
Note that the subindex $i\bullet$ indicates a sum over the set $\cZ\setminus \{i\}$, while the subindex $\Sigma$ indicates a sum over the set $\{1,\ldots,n\}$.

The mean load of the system is defined as $\rho \defn \lambda/\mu_\Sigma$. In the heavy-traffic limit, we take $\rho\uparrow 1$, or more specifically we take $\lambda \uparrow \mu_{\Sigma}$ while $[\mu_{ij}]_{j=1,\ldots,n}$ and $[\alpha_{i i'}]_{i'\in\cZ}$ stay constant. In the rest of this paper, we refer to the load-balancing system described above as the \textit{JSQ system}.

As a notational convenience, we define the \emph{limiting arrival-rate} $\lambda_j^*$ to server $j$ to equal $\mu_j$, the average service rate at state $j$. We define the corresponding limiting arrival-rate vector $\vlambda^* = (\lambda_j^*)_{j=1}^n$, where $\lambda^*_\Sigma \defn \sum_{j=1}^n \lambda_j^*= \mu_\Sigma$.
We also define the \emph{ideal arrival-rate} vector $\vlambda \defn \vlambda^* \rho$. In the heavy-traffic limit, $\vlambda$ linearly approaches $\vlambda^*$.

For ease of notation, we define the parameter $\epsilon\defn 1-\rho$. Notice that for all $\epsilon>0$, the Markov chain $\{Z(t),\vq(t))\}$ is positive recurrent and has a stationary distribution. We add a tilde on top of the random variables and vectors to denote steady state. For example, $\tvq$ represents a random vector that follows the stationary distribution of $\vq(t)$, and $\tq_j$ is its $j\tth$ element. 

Observe that, as $\epsilon \downarrow 0$, the system approaches null recurrence. As $\epsilon\downarrow 0$, the system receives arrivals at a higher rate and, consequently, JSQ tries to equalize all the queue lengths by sending new arrivals to the shortest queue. This phenomenon is known as State Space Collapse (SSC) and is key to our proof.

The main result of this paper is for the JSQ system described above. Our proof technique involves several steps, so we utilize a single-server queue to illustrate the key novelty of our proof approach. Observe that letting $n=1$, the load-balancing system described above fits a single-server queue. For this system, we omit the second index of the service rate, that is, we use $\mu_i$ (and not $\mu_{i1}$) to denote the service rate when $Z(t)=i$.

A key ingredient in our paper is the Poisson equation, which we define below. 
\begin{definition}\label{def:poisson-eqn}
    Consider a function $f:\cZ\to \bR$ and let $\fbar=\E{f(Z)}$. There exists a function $V_f:\cZ\to \bR$ such that
    \begin{align}\label{eq:poisson-eq-def}
        V_f(i) &= \dfrac{f(i) - \fbar}{\alpha_{i,\bullet}} + V_f^+(i), \\
    \text{where }V_f^+(i) &= \sum_{i'\in\cZ,i'\neq i} \frac{\alpha_{i,i'}}{\alpha_{i,\bullet}} V_f(i').
    \nonumber
    \end{align}
    \cref{eq:poisson-eq-def} is known as the \underline{Poisson equation for $f$}.
\end{definition}

Intuitively, $V_f^+(i)$ represents the value of $V_f$ one transition after state $i$. Note that the solution to the Poisson equation $V_f(i)$ is defined up to additive constants only. Our results hold for any choice of this constant. As shown in our proofs, whenever we obtain a term that depends on the solution to the Poisson equation, it is multiplied by $\epsilon$. Hence, the difference between any pair of solutions is negligible in heavy traffic.

Finally, for consistency with the literature on the transform method, we use the drift notation instead of generators. We properly define the drift of a function over a continuous-time Markov chain below, using the notation from \citet{HL-Mag-2022-JSQ-Many-Server}.
\begin{definition}\label{def:Drift}
    Let $\{X(t):t\in\bR_+\}$ be a CTMC with countable state space $\cX$ and transition rate matrix $G^X$. For a function $V:\cX\to \bR_+$ and any $x\in\cX$, define the \underline{drift of $V$ at $x$} as
    \begin{align*}
        \Delta V(x) \defn \sum_{x'\in\cX, x'\neq x} G^X_{x,x'}\left(V(x')-V(x)\right).
    \end{align*}
    We say that we \underline{set the drift of $V$ to zero in steady state} when we use the property $\E{\Delta V(X(t))}=0$ under stationary distribution, provided that $\sup_{x\in\cX}|G_{x,x}^X|<\infty$ and $\E{V(X(t))}<\infty$ in steady state. 
\end{definition}

\subsection{Main result}

Our main result is the asymptotic distribution of the individual scaled queue lengths $\epsilon q_j$ in the JSQ system described in \cref{sec:model}. Specifically, we show that each individual scaled queue length converges to an exponential random variable with the same mean if the arrival rate is sufficiently large in every state of the modulating Markov chain. Further, we show the rate of convergence. We specify the result and the conditions below.

The theorem provides the heavy-traffic distribution of the individual queue lengths, and a key step in the proof is SSC to a one-dimensional system. JSQ works towards equalizing the queue lengths only through routing decisions at arrivals. Accordingly, our proof requires the arrival rate to be sufficiently high so that this balancing mechanism dominates service-induced fluctuations. As a result, our condition depends on both the arrival rates and the heterogeneity of the servers’ rates across states of the modulating process. Specifically, the following condition guarantees SSC.
\begin{definition}\label{def:condition_lambda_i-mu_i}
    Consider a Markov-modulated JSQ system as defined in \cref{sec:model}. For each $i\in\cZ$, define $\mu_{i\min}=\min\{\mu_{ij}: j\in\{1,\ldots,n\}\}$. 
   The system satisfies the \underline{SSC load condition} if  
    \begin{align*}
        \lambda_i > \mu_{i\Sigma} - n \mu_{i\min}\qquad \forall i\in\cZ.
    \end{align*}
\end{definition}

For each $i\in\cZ$, define the following functions
\begin{align}\label{eq:def-h-k-ell-jsq}
    h(i) = \mu_{i\Sigma} -\lambda_i,\quad k(i)= (\mu_{i\Sigma} - \lambda_i) V_{h}(i), \quad \& \quad \ell(i) = \mu_{i\Sigma},
\end{align}
where $V_{h}(i)$ is the solution to the Poisson equation of $h(i)$. We define $V_{k}(i)$ and $V_{\ell}(i)$ equivalently for $k(i)$ and $\ell(i)$, respectively. 

\begin{theorem}\label{thm:jsq-vq-convergence}
    Consider a JSQ system with Markov-modulated arrival and service rates, as described in \cref{sec:model}. Suppose that the system satisfies the SSC load condition. Let $V_{h}(i)$, $V_{k}(i)$ and $V_{\ell}(i)$ be a solution to the Poisson equation for $h(i)$, $k(i)$ and $\ell(i)$ defined in \cref{eq:def-h-k-ell-jsq}, respectively. Suppose $\E{|V_{h}(i)|^{1+\eta}}<\infty$ for some $\eta>0$, $\E{|V_{k}(i)|^{1+\zeta}}<\infty$ for some $\zeta>0$, and $\E{|V_{\ell}(i)|^{1+\beta}}<\infty$ for some $\beta>0$. 
    
    Then, for all $j\in\{1,\ldots,n\}$ and all $s>0$, the Laplace transform of $\epsilon \tq_j$ satisfies
    \begin{align}\label{eq:mgf-q_j}
        \E{e^{-s\epsilon \tq_j}} = \frac{1}{1 + \frac{s}{n}\left(1 + \frac{k^*}{\mu_\Sigma}\right)} + O\left(\epsilon^{1-\frac{1}{1+\min\{1,\eta\}}}\right),
    \end{align}
    where $k^* = \lim_{\epsilon \downarrow 0} \E{k(i)}$. 
    As a consequence, $\epsilon \tq_j$ converges in distribution to an exponential random variable with mean $\frac{1}{n}(1 + k^*/\mu_\Sigma)$.
    Moreover, $\epsilon \tvq\Longrightarrow \Upsilon \vone$, where $\Upsilon$ is an exponential random variable with mean $\frac{1}{n}(1+\frac{k^*}{\mu_\Sigma})$.
\end{theorem}

Our result shows that $\epsilon \tq_j$ converges to an exponential random variable if the SSC load condition is satisfied. This condition ensures that, in each state of the modulating Markov chain, the arrival rate is sufficiently large. In this paper, we don't require any conditions on the mixing time of $\{Z(t)\}_t$. Instead, we make sure that the system has sufficiently high load in every state, so that JSQ can balance the queue lengths. In particular, the lower bound for the total arrival rate $\lambda_i$ represents the range between the total service rate $\mu_{i\Sigma}$ and the worst-case service rate, that is, the total service rate if all the servers had the smallest rate in state $i$. If we rearrange the terms in this condition, we obtain 
\begin{align*}
    \rho_i \defn \frac{\lambda_i}{\mu_{i\Sigma}} > 1 - \frac{n\mu_{i\min}}{\mu_{i\Sigma}}.
\end{align*}
Note that for homogeneous servers, the right-hand side is 0. Hence, this condition also establishes a minimum arrival rate so that JSQ will balance the queue lengths for the given heterogeneity of the servers.

Additionally, we provide the asymptotic Laplace transform of the scaled queue lengths $\epsilon q_j$ as well as the rate of convergence for small $\epsilon$. The rate of convergence specifies the error between the limiting value, which corresponds to the Laplace transform of an exponential random variable, and the pre-limit value $\E{e^{-s\epsilon \tq_j}}$. 

As we demonstrate in the proof, this error term arises when applying the transform method to the scaled average queue lengths $\epsilon \tq_\Sigma/n$. In this step, we analyze the generator of the function $\varphi_s(\vq,i)=e^{-s\epsilon q_\Sigma/n}$. In the traditional transform method (where the parameters are fixed), a key challenge is dealing with the service rate, which depends on which queues are nonempty. Hence, we use SSC to approximate the $n$-dimensional system by a single-server queue. This approximation yields an error of $O(\sqrt{\epsilon})$. 

One of our contributions is generalizing the transform method to Markov-modulated parameters, and we use the Poisson equation to handle the dependence between the vector of queue lengths and the Markov-modulated arrival and service rates. The use of the Poisson equation yields an error of $O(\epsilon^{1-\frac{1}{1+\eta}})$. Hence, the error term $\epsilon^{1-\frac{1}{1+\min\{1,\eta\}}}$ comes from analyzing the drift of $\varphi_s(\vq,i)=e^{-s\epsilon q_\Sigma/n}$ and dealing with the product of dependent random variables.

The proof of \cref{thm:jsq-vq-convergence} is provided in \cref{sec:roc}. We break the proof into a series of smaller results that are proved in \cref{sec:jsq}, and are of independent interest.

In \cref{sec:empirical} we provide a simulated example to build an intuitive meaning to our result. Specifically, we show that our result holds under various mixing time assumptions. 

\section{Proof Roadmap and Methodological Contributions}
\label{sec:proof-overview}

In this section, we discuss an outline of our proof, starting with the key ingredients in the transform method.

% 1. What is Transform method --> Show M/M/1
% 2. What happens with Markov-modulated parameters: cross-terms between modulating Markov chain and queue lengths. Not independent anymore!
% 3. Use of Poisson eqn to deal with cross-terms via computation of covariance
% 4. SSC discussion
% 5. How errors arise and why they're negligible in heavy traffic

\paragraph{The transform method:}
The transform method is a drift-based approach to compute the Laplace transform (or moment generating function) of the queue lengths. A key step is in handling the so-called \textit{unused service terms}, which are cross-terms arising from the dynamics of the multiple queue lengths. Specifically, they arise because a job can leave the system only when the corresponding queue is nonempty. Managing these cross terms is the key difficulty in analyzing the distribution of queueing systems. In the diffusion limits approach, for example, the scaled queue lengths converge to \textit{Reflected} Brownian motion (and not simply Brownian motion) because of the unused-service terms. 

The transform method neatly deals with this challenge by observing that for any function $\phi(q)$ of the queue lengths in a single-server queue, we have $\phi(q)\ind{q=0}=\phi(0)\ind{q=0}$. 
This key observation cleans all the cross terms, and the remaining of the proof approach relies on SSC (which we discuss below). To make this paper self-contained, we demonstrate the transform method approach for an $M/M/1$ queue in Appendix~\ref{app:transform-MM1}. For an $M/M/1$ queue, the transform method yields the stationary distribution for all traffic intensities in the capacity region. However, for more intricate models, one obtains bounds, which can then be used to compute the Laplace transform in heavy traffic.

\paragraph{Need for novel techniques:}
In systems with constant parameters, the only cross-terms are due to the unused-service, as stated above. These explicitly depend on the queue length(s) and hence, they can be handled algebraically. When the parameters are Markov modulated, there are new cross terms that depend on the arrival and service rates, and the queue lengths. However, handling these new cross terms is more challenging because the dependence between the parameters and the queue lengths is not explicit. 

We tackle this challenge via the Poisson equation. Specifically, we use the solution to the Poisson equation to characterize the covariance between the Markov-modulated parameters and the queue lengths. 

\paragraph{State-Space Collapse:}
For a single-server queue, the use of the transform method and the Poisson equation are sufficient to complete the proof (as shown in \cref{sec:ssq}). However, systems with multiple queues are more challenging as the queue lengths typically depend on each other. In the JSQ system studied here, the queue lengths are tied because the routing policy sends new arrivals to the shortest queue. Characterizing the dependence among queues becomes intractable for small values of $n$. We make the heavy-traffic approximation and show that in heavy traffic, all the queue lengths are similar. Therefore, studying a single-server queue yields a good approximation to the multi-queue system.

\paragraph{Error terms:}
As part of our methodology, we obtain error terms that are negligible in heavy traffic. These arise from our approach to handling cross terms in the drift of the exponential test function. Hence, we have error terms arising from our computation of the covariance between the modulating Markov chain and the queue lengths via the Poisson equation, as well as from SSC. An essential aspect of our research approach is the careful handling of these error terms to obtain asymptotically tight bounds. Our theorems characterize the magnitude of these error terms with respect to the heavy-traffic parameter $\epsilon$ and show that they are negligible as $\epsilon\downarrow 0$.

\section{Single-Server Queue Warm-Up}
\label{sec:ssq}

The proof of \cref{thm:jsq-vq-convergence}, our main result, involves several steps, including the use of the Poisson equation, SSC and the convergence of the ``collapsed'' queue lengths to the exponential distribution. One of our key methodological contributions is the use of the Poisson equation to address the dependence between the arrival and service rates, and the queue length. To isolate this contribution of our paper, we showcase the approach in the simplest system: a single-server queue. Many of the ideas presented in this section will reappear in the proof of \cref{thm:jsq-vq-convergence}.

We start with a key theorem relating the solution to the Poisson equation to the Laplace transform of the scaled queue lengths.

\subsection{Poisson Equation for the Single-Server Queue}

The following theorem provides an expression for the covariance between the queue length and functions of the modulating Markov chain, and is one of the key methodological contributions of this paper. Recall that $q(t)$ represents the number of jobs in the system at time $t$, and $(q(t),Z(t))$ is a positive recurrent CTMC for all $\epsilon=1-\rho=1-\lambda/\mu>0$. We use $\tq$ for a random variable that follows the stationary distribution of the queue length.

\begin{theorem}\label{thm:ssq-poisson-lemma}
    Let $f:\cZ\to \bR$ be any function such that a solution $V_f(i)$ to the Poisson equation (as in \cref{def:poisson-eqn}) exists and satisfies $\E{|V_f(i)|^{1+\xi}}<\infty$ for some $\xi>0$. Then, for all $s>0$,
    \begin{align}\label{eq:ssq-poisson-lemma-eqn}
        &\text{Cov}(e^{-s \epsilon \tq}, f(i)) = (e^{-s\epsilon}-1) \E{e^{-s\epsilon \tq} V_f(i)\left(\lambda_i-\mu_i\right)} + O\left(\epsilon^{2-\frac{1}{1+\xi}}\right).
    \end{align}
\end{theorem}
Notice that $\E{|V_f(i)|^{1+\xi}}<\infty$ implies $\E{V_f(i)e^{-s\epsilon \tq}}<\infty$ because $0<e^{-s\epsilon \tq}\leq 1$ for all $s,\epsilon>0$.

When the system parameters are constant (as opposed to Markov-modulated), the function $f(i)=f$ is constant and consequently, $f$ is independent of the queue lengths. In such a case, the solution to the Poisson equation $V_f(i)$ is constant and \cref{eq:ssq-poisson-lemma-eqn} is trivial. Therefore, the right-hand side of \cref{eq:ssq-poisson-lemma-eqn} represents the error induced by assuming that $f(i)$ and $e^{-s\epsilon \tq}$ are uncorrelated. As mentioned above, a significant challenge in computing the Laplace transform of the scaled queue lengths, $\epsilon \tq$, is that the arrival and service rates are correlated with the queue length. 

Using the definition of covariance in \cref{eq:ssq-poisson-lemma-eqn} and rearranging terms, we obtain that
\begin{align}\label{eq:ssq-poisson-lemma-eqn-to-use}
    \E{e^{-s\epsilon \tq}f(i)} = \E{e^{-s\epsilon \tq}}\fbar + (e^{-s\epsilon}-1) \E{e^{-s\epsilon \tq} V_f(i)\left(\lambda_i-\mu_i\right)} + O\left(\epsilon^{2-\frac{1}{1+\xi}}\right).
\end{align}
Hence, \cref{thm:ssq-poisson-lemma} provides the error of approximating the Markov-modulated system by one where the parameters and queue lengths are uncorrelated. This error term is key in our approach. In the rest of the paper, we use the version of \cref{thm:ssq-poisson-lemma} presented in \cref{eq:ssq-poisson-lemma-eqn-to-use}.

The proof of \cref{thm:ssq-poisson-lemma} is similar to its equivalent for the JSQ-system (see \cref{thm:jsq-poisson-lemma}), so we omit it.

\subsection{Computing the MGF of the scaled queue length}

Before stating the main theorem of this section, define the following functions:
\begin{align}\label{eq:ssq-def-h-k-ell}
    \hat{h}(i) = \mu_{i}-\lambda_i,\quad \hat{k}(i) = V_{\hat{h}}(i)\left(\mu_i-\lambda_{i}\right),\quad\&\quad \hat{\ell}(i)=\mu_i.
\end{align}

\begin{theorem}\label{thm:ssq-heavy-traffic-distribution}
    Consider a single-server queue with Markov-modulated arrival and service rates, as described in \Cref{sec:model}, that is, a JSQ model with $n=1$. Let $V_{\hat{h}}(i)$, $V_{\hat{k}}(i)$, and $V_{\hat{\ell}}(i)$ be a solution to the Poisson equation for the functions $\hat{h}(i)$, $\hat{k}(i)$ and $\hat{\ell}(i)$ defined in \cref{eq:ssq-def-h-k-ell}, respectively. 
    Suppose $\E{|V_{\hat{h}}(i)|^{1+\eta}}<\infty$ for some $\eta>0$, $\E{|V_{\hat{k}}(i)|^{1+\zeta}}<\infty$ for some $\zeta>0$, and $\E{|V_{\hat{\ell}}(i)|^{1+\beta}}<\infty$ for some $\beta>0$. Then, for all $s>0$, the Laplace transform of $\epsilon \tq$ satisfies
    \begin{align*}
        \E{e^{-s\epsilon \tq}} &= \frac{1}{1+s(1+\hat{k}^*/\mu)} + O\left(\epsilon^{1-\frac{1}{1+\eta}}\right),
    \end{align*}
    where $\hat{k}^* = \lim_{\epsilon\downarrow 0}\E{\hat{k}(i)}$. 
    This implies that $\epsilon \tq\Longrightarrow \Psi$, where $\Psi$ is an exponential random variable with mean $1+\hat{k}^*/\mu$.
\end{theorem}

\begin{proof}{Proof of \cref{thm:ssq-heavy-traffic-distribution}.}
Let $s>0$ and define the test function $\varphi_s(q,i) = e^{-s \epsilon q}$. We start by computing its drift:
\begin{align*}
    \Delta \varphi_s (q,i) &= \lambda_i \left(e^{-s\epsilon (q+1)}-e^{-s\epsilon q} \right) + \ind{q>0} \mu_i \left(e^{-s\epsilon (q-1)}-e^{-s\epsilon q} \right) \\
    &= (e^{-s\epsilon} -1) \left( e^{-s\epsilon q}(\lambda_i - \mu_i e^{s\epsilon}) + \mu_i e^{s\epsilon}e^{-s\epsilon q}\ind{q=0} \right),
\end{align*}
where we rearranged terms using the fact that $\ind{q>0}=1-\ind{q=0}$.

Since $\E{\varphi_s(\tq,i)}\leq 1<\infty$, we can set its drift to zero. Taking expectation with respect to the stationary distribution, using that $\E{\Delta \varphi_s(\tq,i)}=0$ in the expression above and rearranging terms, we obtain:
\begin{align}\label{eq:ssq-service-mgf-to-solve}
    \E{e^{-s\epsilon \tq}(\mu_i e^{s\epsilon} - \lambda_i)} =  e^{s\epsilon} \E{\mu_i e^{-s\epsilon \tq}\ind{\tq=0}}.
\end{align}

To compute the right-hand side of \eqref{eq:ssq-service-mgf-to-solve}, observe that $e^{-s\epsilon \tq}\ind{\tq=0}=\ind{\tq=0}$ by definition of the indicator function. Then, it remains to compute $\E{\mu_i\ind{\tq=0}}$, which we do by setting $s=0$ in \eqref{eq:ssq-service-mgf-to-solve}. Specifically, noticing that $\E{\mu_i-\lambda_i}=\mu\epsilon$, we obtain:
\begin{align}\label{eq:ssq-service-prob-empty-queue}
    \E{\mu_i\ind{\tq=0}} = \mu\epsilon.
\end{align}

The left-hand side of \eqref{eq:ssq-service-mgf-to-solve} needs more work, as there are cross terms between the modulating Markov chain and the queue length. Specifically, $\lambda_i$ and $\mu_i$ are not independent of $q$. For ease of notation, define
\begin{align*}
    \hat{\cT} \defn e^{-s\epsilon \tq}\left(\mu_i e^{s\epsilon}-\lambda_i\right),
\end{align*}
and observe
\begin{align*}
    \hat{\cT}
    &\stackrel{(a)}{=} e^{-s\epsilon \tq} \left(\mu_i (1+s\epsilon+O(\epsilon^2) ) - \lambda_i\right) \stackrel{(b)}{=} e^{-s\epsilon \tq} \left(\mu_i-\lambda_i\right) + e^{-s\epsilon \tq}\mu_i\epsilon s + O(\epsilon^2),
\end{align*}
where $(a)$ holds by taking the Taylor expansion of $e^{s\epsilon}$ up to first order; and $(b)$ by rearranging terms and because $\mu_i<\mu_{\max}$ and $e^{-s\epsilon \tq}\leq 1$ for all $s>0$.  Then,
\begin{align}\label{eq:expected-T}
    \E{\hat{\cT}} = \E{e^{-s\epsilon \tq}(\mu_i-\lambda_i)} + s\epsilon \E{e^{-s\epsilon \tq}\mu_i} + O(\epsilon^2).
\end{align}

We now bound each of the first two terms using \cref{thm:ssq-poisson-lemma} to handle cross terms between the Markov-modulated parameters $\lambda_i,\mu_i$ and the queue length. 

For the first term, we set $f(i)=\hat{h}(i)=\mu_i-\lambda_i$ (as defined above). Noticing that $\E{\hat{h}(i)}=\mu-\lambda=\mu\epsilon$, we obtain
\begin{align}
    \E{e^{-s\epsilon \tq}(\mu_i-\lambda_i)} 
    &= \E{e^{-s\epsilon \tq}}\mu \epsilon + (e^{-s\epsilon}-1)\E{e^{-s\epsilon \tq}V_{\hat{h}}(i)(\lambda_i-\mu_i)} + O\left(\epsilon^{2 - \frac{1}{1+\eta}}\right) \nonumber \\
    &\stackrel{(a)}{=} \mu\epsilon \E{e^{-s \epsilon \tq}} + s\epsilon \E{e^{-s\epsilon \tq}V_{\hat{h}}(i)(\mu_i-\lambda_i)} + O\left(\epsilon^{2 - \frac{1}{1+\eta}}\right), \label{eq:ssq-poisson-h} 
\end{align}
where 
$(a)$ holds after taking Taylor expansion of $e^{-s\epsilon}$ up to second order, and noticing that $\E{e^{-s\epsilon \tq}V_{\hat{h}}(i)(\lambda_i-\mu_i)}$ is finite because $e^{-s\epsilon \tq}\leq 1$, $\lambda_i\leq \lambda_{\max}$, $\mu_i<\mu_{\max}$ and $\E{|V_h(i)|}<\infty$ by assumption of the theorem. The $O(\epsilon^2)$ term arising from the Taylor expansion goes into the $O\left(\epsilon^{2 - \frac{1}{1+\eta}}\right)$ term. 

We now compute the second term in \eqref{eq:ssq-poisson-h} using \cref{thm:ssq-poisson-lemma} with $f(i)=\hat{k}(i)=V_{\hat{h}}(i)(\mu_i-\lambda_i)$. We obtain
\begin{align}
    & \E{e^{-s\epsilon \tq}V_{\hat{h}}(i)(\mu_i-\lambda_i)} \nonumber \\
    &= \E{e^{-s\epsilon \tq}}\E{V_{\hat{h}}(i)(\mu_i-\lambda_i)} + (e^{-s\epsilon}-1)\E{e^{-s\epsilon \tq}V_{\hat{k}}(i)(\lambda_i-\mu_i)} + O(\epsilon^{2-\frac{1}{1+\zeta}}) \nonumber \\
    &\stackrel{(*)}{=} \E{e^{-s\epsilon \tq}}\E{\hat{k}(i)} + O(\epsilon), \label{eq:ssq-poisson-k}
\end{align}
where $(*)$ uses that $k(i)=V_h(i)(\mu_i-\lambda_i)$ and bounds the expectation multiplying $(e^{-s\epsilon}-1)$ following a similar argument as the computation of equality $(a)$ in \eqref{eq:ssq-poisson-h}.

Using \eqref{eq:ssq-poisson-k} in \eqref{eq:ssq-poisson-h}, we obtain
\begin{align}\label{eq:ssq-poisson-hk}
    \E{e^{-s\epsilon \tq}(\mu_i-\lambda_i)} 
    &= \mu \epsilon \E{e^{-s\epsilon \tq}} + s\epsilon \E{e^{-s\epsilon \tq}}\E{\hat{k}(i)} + O\left(\epsilon^{2 - \frac{1}{1+\eta}}\right),
\end{align}
which finalizes the computation of the first term of \eqref{eq:expected-T}. We now compute the second expectation, namely $\E{e^{-s\epsilon \tq}\mu_i}$. To this end, we use \cref{thm:ssq-poisson-lemma} with $f(i)=\ell(i)=\mu_i$ and obtain
\begin{align}
    \E{e^{-s\epsilon \tq}\mu_i} 
    &= \E{e^{-s\epsilon \tq}}\mu + (e^{-s\epsilon}-1)\E{e^{-s\epsilon \tq} V_\ell(i)(\lambda_i-\mu_i)} + O\left(\epsilon^{2 - \frac{1}{1+\beta}}\right) \nonumber \\
    &\stackrel{(*)}{=} \mu \E{e^{-s\epsilon \tq}} + O(\epsilon), \label{eq:ssq-poisson-l}
\end{align}
where $(*)$ follows by an argument similar to the computation of \eqref{eq:ssq-poisson-k}.

Using \eqref{eq:ssq-poisson-hk} and \eqref{eq:ssq-poisson-l} in \eqref{eq:expected-T}, we obtain
\begin{align}
    \E{\hat{\cT}} &= \mu \epsilon \E{e^{-s\epsilon \tq}} + s\epsilon \E{e^{-s\epsilon \tq}}\E{k(i)} + s\epsilon\mu \E{e^{-s\epsilon \tq}} + O\left(\epsilon^{2 - \frac{1}{1+\eta}}\right) \nonumber \\
    &= \mu\epsilon \E{e^{-s\epsilon \tq}}\left(1 + s\left( 1 + \frac{1}{\mu}\E{\hat{k}(i)} \right) \right) + O\left(\epsilon^{2 - \frac{1}{1+\eta}}\right). \label{eq:ssq-expected-T-solved}
\end{align}

Now we put everything together, that is, we use \eqref{eq:ssq-service-prob-empty-queue} and \eqref{eq:ssq-expected-T-solved} in \eqref{eq:ssq-service-mgf-to-solve}. Rearranging terms algebraically, we obtain
\begin{align*}
    \E{e^{-s\epsilon \tq}} &= \frac{1}{1 + s\left( 1 + \frac{\E{\hat{k}(i)}}{\mu} \right)} + O\left(\epsilon^{1-\frac{1}{1+\eta}}\right).
\end{align*}

The last step is to argue that the denominator has a $\hat{k}^*$ instead of $\E{\hat{k}(i)}$. To compare $\hat{k}^*$ and $\E{\hat{k}(i)}$, recall from \cref{sec:model} that $\lambda = \mu_{\Sigma} (1-\epsilon)$, while $\mu_{i}$ and $[\alpha_{i i'}]_{i'\in\cZ}$ are constant with respect to $\epsilon$.
As a result, $\hat{h}(i) = \mu_i - \lambda_i$ is a linear function of $\epsilon$, of the form $c_1 + c_2 \epsilon$ for some constants $c_1$ and $c_2$. Note that $V_{\hat{h}}(i)$ is also a linear function of $\epsilon$ for each $i$, as it is the solution to a system of linear equations with constant coefficients and linear additive terms.
Thus, $\hat{k}(i) = V_{\hat{h}}(i) (\mu_i - \lambda_i)$ is a quadratic function of $\epsilon$, of the form $c_3 + c_4 \epsilon + c_5 \epsilon^2$.
In particular, we find that $\hat{k}^* = c_3 = \lim_{\epsilon\downarrow 0}\E{\hat{k}(i)}$, and thus $|\hat{k}^* - \E{\hat{k}(i)}| = |c_4 \epsilon + c_5 \epsilon^2| = O(\epsilon)$. The replacement of $\E{\hat{k}(i)}$ by $\hat{k}^*$ is thus absorbed by the existing error term.
\Halmos
\end{proof}

\section{Generalization to JSQ System}
\label{sec:jsq}

In \cref{sec:ssq}, we introduced our methodology for computing the Laplace transform of the queue length when the arrival and service rates are Markov-modulated, and hence, they are not independent of the queue length. The key ingredient was to use the Poisson equation to handle cross-terms between the Markov-modulated parameters and the queue length. In this section, we integrate this step with SSC and prove the main result of the paper. In particular, the results and methodology presented in \cref{sec:ssq} intuitively showcase the analysis of the ``collapsed'' system in heavy traffic. In this section, we additionally track and carefully handle all the error terms arising from SSC.

\subsection{State Space Collapse}
\label{sec:ssc}

In the JSQ system, SSC indicates that as the traffic intensity increases, all the queue lengths become approximately equal. Intuitively, this phenomenon occurs because JSQ assigns new arrivals to the shortest queue, thereby reducing the difference between the longest and shortest queues in the system. 

The following proposition formally presents this SSC result by showing that the magnitude of the expected difference between the vector of queue lengths in steady state $\tvq$ and its projection on the vector $\vone$ is bounded by constants that do not depend on $\epsilon$. This demonstrates SSC as described above, because all the queue lengths increase when $\epsilon\downarrow 0$. Hence, a bounded error becomes insignificant in the heavy-traffic limit. 

Given a vector $\vq\in\bZ^n_+$, define
\begin{align}\label{eq:def-qpar-qperp}
    \vq_\parallel \defn \vone \Big(\dfrac{1}{n} \sum_{j=1}^n q_j\Big), \qquad \vq_\perp \defn \vq - \vq_\parallel,
\end{align}
and let $q_{\parallel j}$ and $q_{\perp j}$ denote the $j\tth$ element of $\vq_\parallel$ and $\vq_\perp$, respectively.

\begin{proposition}\label{prop:ssc}
    Consider a JSQ system as described in \cref{sec:model}. If the SSC load condition in \cref{def:condition_lambda_i-mu_i} is satisfied, then there exist finite constants $0<a<1$, $b>0$ and $v>0$ such that for any nonnegative integer $p$, we have
    \begin{align}\label{eq:qperp-bound-tail-prob}
        P\left(\|\tvq_\perp\|>b+v p\right) \leq a^{p+1}.
    \end{align}
    Consequently, for any positive integer $r$, there exists a constant $C_r$ such that
    \begin{align}\label{eq:qperp-bound-poly}
        \E{\|\tvq_\perp\|^r}\leq C_r.
    \end{align} 
    Further, there exists $\Theta>0$ and a finite constant $C_{\exp}>0$ such that
    \begin{align}\label{eq:qperp-bound-expo}
        \E{e^{\theta \|\tvq_\perp\|}}\leq C_{\exp}, \qquad \forall |\theta| < \Theta.
    \end{align}
\end{proposition}

The proof of \cref{prop:ssc} is standard for heavy-traffic analysis, so we present it in Appendix~\ref{app:ssc-proof}. The main idea is to bound the drift of the test function $W_\perp (\vq)\defn \|\vq_\perp\|$. The proof is mostly algebraic, but there is a key step where we estimate the arrival rate to each queue after routing. We do so by defining the hypothetical arrival rate vector in state $i$ as
\begin{align*}
    \vlambda_i' \defn \vmu_i - \delta_i \mu_{i\Sigma}\vone,
\end{align*}
with $\delta_i \defn \left(1 - \lambda_i/\mu_{i\Sigma} \right)/n$. The SSC load condition ensures that all the elements of the hypothetical arrival rate vector $\vlambda_i'$ are positive, that is, it ensures that it is achievable for JSQ dispatching to match the hypothetical arrival rates for each server in each modulation state.

\subsection{Poisson Equation for the JSQ System}

In this section, we generalize \cref{thm:ssq-poisson-lemma} to the JSQ system and present the proof. We use this theorem to handle cross terms between the queue lengths and the Markov-modulated parameters.

\begin{theorem}\label{thm:jsq-poisson-lemma}
    Consider the JSQ system described in \cref{sec:model} and the Poisson equation in \cref{def:poisson-eqn}. Let $f:\cZ\to \bR$ be any function such that a solution $V_f(i)$ to the Poisson equation exists and satisfies $\E{|V_f(i)|^{1+\xi}}<\infty$ for some $\xi>0$. Then, for all $s>0$,
    \begin{align*}
        & Cov(e^{-s\epsilon \tq_\Sigma},f(i)) 
        %= \E{e^{-s \epsilon \tq_\Sigma} f(i)} - \E{e^{-s \epsilon \tq_\Sigma}} \fbar 
        =(e^{-s\epsilon}-1) \E{e^{-s\epsilon \tq_\Sigma} V_f(i)\left(\lambda_i-\mu_{i\Sigma}\right)} + O\left(\epsilon^{2-\frac{1}{1+\xi}}\right).
    \end{align*}
\end{theorem}

Note that we write the theorem with respect to the Laplace transform of the scaled total queue length $\epsilon \tq_\Sigma$. However, it holds for all $s>0$. In particular, it holds for $s=\hat{s}/n$. Hence, one can immediately extend the result for the average queue length $\tq_\Sigma/n$, which corresponds to any element of $\tvq_\parallel$. We now provide the proof.

\begin{proof}{Proof of \cref{thm:jsq-poisson-lemma}.}

Recall the definition of covariance:
\begin{align*}
    Cov(e^{-s\epsilon\tq_\Sigma},f(i)) = \E{e^{-s \epsilon \tq_\Sigma} f(i)} - \E{e^{-s \epsilon \tq_\Sigma}} \bar{f}.
\end{align*}
Then, we compute bounds for the right-hand side. To do so, we start by rewriting the Poisson equation \eqref{eq:poisson-eq-def} as follows
\begin{align*}
    f(i)-\fbar = \alpha_{i,\bullet} \left(V_f(i) - V^+_f(i) \right).
\end{align*}
Multiply both sides by $e^{-s \epsilon \tq_\Sigma}$ and apply expectation under stationary distribution to obtain
\begin{align}\label{eq:jsq-service-lemma-1-proof-1}
    \E{e^{-s \epsilon \tq_\Sigma} f(i)} - \E{e^{-s \epsilon \tq_\Sigma}} \bar{f} = \E{e^{-s \epsilon \tq_\Sigma} V_f(i)\alpha_{i,\bullet}} - \E{e^{-s\epsilon \tq_\Sigma}V_f^+(i)\alpha_{i,\bullet}}.
\end{align}

Let's now compute the drift of $g_s(\vq,i) = e^{-s\epsilon q_\Sigma} V_f(i)$. We obtain
\begin{align*}
    & \Delta g_s(\vq,i) \\
    &= e^{-s\epsilon q_\Sigma} V_f(i)\Big(\lambda_i (e^{-s\epsilon}-1) + (e^{s\epsilon}-1) \sum_{j=1}^n(\mu_{ij} \ind{q_j>0})\Big)
    + e^{-s\epsilon q_\Sigma} \big( (\sum_{i'\in\cZ, i'\neq i} V_f(i') \alpha_{ii'}) - V_f(i)\alpha_{i,\cdot}\big).
\end{align*}

Observe that $\E{\hat{g}_s(\vq,i)}<\infty$ because $0<e^{-s\epsilon \tq_\Sigma}\leq 1$ and the $1+\xi$ moment of $V_f(i)$ is finite. Then, we can set its drift to zero: $\E{\Delta \hat{g}_s(\tvq,i)}=0$. Taking expectation with respect to the stationary distribution, and reorganizing terms, we obtain
\begin{align}
    & \E{e^{-s\epsilon \tq_\Sigma} \Big( V_f(i)\alpha_{i,\bullet} - \sum_{i'\in\cZ, i'\neq i} V_f(i') \alpha_{ii'} \Big)} \nonumber \\
    &= \E{e^{-s\epsilon \tq_\Sigma} V_f(i)\Big(\lambda_i (e^{-s\epsilon}-1) + (e^{s\epsilon}-1) \sum_{j=1}^n(\mu_{ij} \ind{\tq_j>0})\Big)} \nonumber \\
    &= \E{e^{-s\epsilon \tq_\Sigma} V_f(i)\left(\lambda_i (e^{-s\epsilon}-1) + (e^{s\epsilon}-1) \mu_{i \Sigma}\right)} - (e^{s\epsilon}-1) \E{e^{-s\epsilon \tq_\Sigma} V_f(i)\sum_{j=1}^n(\mu_{ij} \ind{\tq_j=0})}.
    \label{eq:jsq.step0}
\end{align}
Note that the left-hand side of \eqref{eq:jsq.step0} matches the right-hand side of \eqref{eq:jsq-service-lemma-1-proof-1}. Hence, it suffices to bound each of the expectations on the right-hand side of \eqref{eq:jsq.step0}. We start with the last term:
\begin{align*}
    \left|\E{e^{-s\epsilon \tq_\Sigma} V_f(i)\sum_{j=1}^n (\mu_{ij} \ind{q_j=0})} \right|
    &\stackrel{(a)}{\leq}\E{e^{-s\epsilon \tq_\Sigma} |V_f(i)|\sum_{j=1}^n (\mu_{ij} \ind{q_j=0})} \\
    &\stackrel{(b)}{\leq} \E{|V_f(i)| \sum_{j=1}^n (\mu_{ij} \ind{\tq_j=0})} \\
    &\stackrel{(c)}{\leq} \E{(\sum_{j=1}^n \mu_{ij}\ind{\tq_j=0})^{1+\frac{1}{\xi}}}^{1-\frac{1}{1+\xi}} \E{|V_f(i)|^{1+\xi}}^{\frac{1}{1+\xi}} \\
    &\stackrel{(d)}{\leq} \mu_{\max}^{\frac{1}{\xi}} \E{\sum_{j=1}^n \mu_{ij}\ind{\tq_j=0}}^{1-\frac{1}{1+\xi}} \E{|V_f(i)|^{1+\xi}}^{\frac{1}{1+\xi}} \\
    &\stackrel{(e)}{=} \epsilon^{1-\frac{1}{1+\xi}} \left(\mu_{\Sigma}^{1-\frac{1}{1+\xi}}\mu_{\max}^\frac{1}{\xi} \E{|V_f(i)|^{1+\xi}}^{\frac{1}{1+\xi}} \right)\\
    &\stackrel{(f)}{=} O\left(\epsilon^{1-\frac{1}{1+\xi}}\right),
\end{align*}
where $(a)$ holds by the triangle inequality; 
$(b)$ holds because $e^{-s\epsilon \tq_\Sigma}\leq 1$; 
$(c)$ holds by Hölder's inequality with $\xi>0$; 
$(d)$ holds because $\mu_{ij}\ind{\tq_j=0}\leq \mu_{ij}\leq \mu_{\max}$; 
$(e)$ holds because $\E{\sum_{j=1}^n \mu_{ij}\ind{\tq_j=0}}=\mu_\Sigma\epsilon$. This can be proved by setting to zero the drift of $g(\vq,i)=q_\Sigma$, and an alternative proof is provided in \cref{jsq-eq:service-prob-empty-queue}. 
Step $(f)$ holds because the constant multiplying $\epsilon^{1-\frac{1}{1+\xi}}$ is finite by assumption.

Substituting into \eqref{eq:jsq.step0}, and observing that $(e^{s\epsilon}-1)=O(\epsilon)$, we obtain
\begin{align}
    & \E{e^{-s\epsilon \tq_\Sigma} \Big( V_f(i)\alpha_{i,\bullet} - \sum_{i'\in\cZ, i'\neq i} V_f(i') \alpha_{ii'} \Big)} \nonumber \\
    &= \E{e^{-s\epsilon \tq_\Sigma} V_f(i)\left(\lambda_i (e^{-s\epsilon}-1) + (e^{s\epsilon}-1) \mu_{i \Sigma}\right)} + O\left(\epsilon^{2-\frac{1}{1+\xi}}\right). \label{eq:jsq.step1}
\end{align}
Now we solve for the first term of \cref{eq:jsq.step1}. We have:
\begin{align}
    & \E{ e^{-s\epsilon \tq_\Sigma}V_f(i) \left(\lambda_i (e^{-s\epsilon}-1) + \mu_i(e^{s\epsilon}-1)\right)} \nonumber \\
    &= 
    (e^{-s\epsilon}-1) \E{e^{-s\epsilon \tq_\Sigma}V_f(i) \left(\lambda_i -\mu_{i \Sigma} e^{s\epsilon} \right)} \nonumber \\
    &\stackrel{(a)}{=} (e^{-s\epsilon}-1) \E{e^{-s\epsilon \tq_\Sigma} V_f(i)\left(\lambda_i-\mu_{i \Sigma}\right)}  - (e^{-s\epsilon}-1)  \E{e^{-s\epsilon \tq_\Sigma} V_f(i)\left(\mu_{i \Sigma} s \epsilon + O(\epsilon^2) \right)} \nonumber \\
    &\stackrel{(b)}{=} (e^{-s\epsilon}-1) \E{e^{-s\epsilon \tq_\Sigma}V_f(i) \left(\lambda_i-\mu_{i \Sigma}\right)}  +O(\epsilon^2), \label{eq:jsq.step2}
\end{align}
where $(a)$ holds by taking the Taylor expansion of $e^{s\epsilon}$ up to second order; 
and $(b)$ uses that $e^{-s\epsilon \tq_\Sigma}\leq 1$, $\mu_{i\Sigma}\leq n\mu_{\max}$ and $\E{|V_f(i)|}<\infty$ to bound the second term.

Putting \eqref{eq:jsq.step1} and \eqref{eq:jsq.step2} together in \eqref{eq:jsq-service-lemma-1-proof-1}, we obtain
\begin{align*}
        & \E{e^{-s \epsilon \tq_\Sigma} f(i)} - \E{e^{-s \epsilon \tq_\Sigma}} \bar{f} =(e^{-s\epsilon}-1) \E{e^{-s\epsilon \tq_\Sigma} V_f(i)\left(\lambda_i-\mu_{i \Sigma}\right)} + O\left(\epsilon^{2-\frac{1}{1+\xi}}\right). \Halmos
    \end{align*}
\end{proof}

\subsection{MGF Results}

In this section, we prove the main result of this paper, \cref{thm:jsq-vq-convergence}.
Our goal is to show that the individual scaled queue lengths $\epsilon \tq_j$ converge to an exponential random variable, for all $j\in\{1,\ldots,n\}$. We repeat the statement below.

\begin{repeattheorem}[Theorem~\ref{thm:jsq-vq-convergence}.]
    Consider a JSQ system with Markov-modulated arrival and service rates, as described in \cref{sec:model}. Suppose that the system satisfies the SSC load condition. Let $V_{h}(i)$, $V_{k}(i)$ and $V_{\ell}(i)$ be a solution to the Poisson equation for $h(i)$, $k(i)$ and $\ell(i)$ defined in \cref{eq:def-h-k-ell-jsq}, respectively. Suppose $\E{|V_{h}(i)|^{1+\eta}}<\infty$ for some $\eta>0$, $\E{|V_{k}(i)|^{1+\zeta}}<\infty$ for some $\zeta>0$, and $\E{|V_{\ell}(i)|^{1+\beta}}<\infty$ for some $\beta>0$. 
    
    Then, for all $j\in\{1,\ldots,n\}$ and all $s>0$, the Laplace transform of $\epsilon \tq_j$ satisfies
    \begin{align*}
        \E{e^{-s\epsilon \tq_j}} = \frac{1}{1 + \frac{s}{n}\left(1 + \frac{k^*}{\mu_\Sigma}\right)} + O\left(\epsilon^{1-\frac{1}{1+\min\{1,\eta\}}}\right),
    \end{align*}
    where $k^* = \lim_{\epsilon \downarrow 0} \E{k(i)}$. 
    As a consequence, $\epsilon \tq_j$ converges in distribution to an exponential random variable with mean $\frac{1}{n}(1 + k^*/\mu_\Sigma)$.
    Moreover, $\epsilon \tvq\Longrightarrow \Upsilon \vone$, where $\Upsilon$ is an exponential random variable with mean $\frac{1}{n}(1+\frac{k^*}{\mu_\Sigma})$.
\end{repeattheorem}
\begin{proof}{Proof of \cref{thm:jsq-vq-convergence}.}

We develop the proof in several steps. First, recall that for all $j$,
\begin{align*}
    \epsilon \tq_j = \epsilon \frac{\tq_\Sigma}{n} + \epsilon \tq_{\perp j}.
\end{align*}
Since the moments of $\tvq_\perp$ are bounded (\cref{prop:ssc}), then $\epsilon \tvq_\perp$ converges to 0 in the mean-squared sense. Hence, $\epsilon \tvq$ and $\epsilon \tvq_\parallel$ converge in distribution to the same limit. 
Intuitively, this means that as $\epsilon$ decreases to 0,
\begin{align*}
    \epsilon \tq_j \approx \epsilon\tq_{\parallel j} = \epsilon \frac{\tq_\Sigma}{n} \quad \forall j\in\{1,\ldots,n\}.
\end{align*}

Hence, we only need the limiting distribution of $\epsilon \tq_\Sigma/n$, which we obtain in \cref{thm:jsq-heavy-traffic-distribution-q_avg}. Finally, we characterize the error of approximating the MGF of $\epsilon \tq_j$ by the transform of $\epsilon \tq_\Sigma/n$ and compute the rate of convergence to the heavy traffic limit in \cref{thm:jsq-ROC}. 
\Halmos
\end{proof}

\subsubsection{Distribution of the Average Queue Lengths}

Similarly to the single-server queue, we repeatedly use the Poisson equation to compute the MGF of the queue lengths. In particular, we use \cref{thm:jsq-poisson-lemma} with the following functions:
\begin{align*}
    h(i) \defn \mu_{i \Sigma}-\lambda_i,\quad k(i) \defn V_{h}(i)\left(\mu_{i \Sigma} - \lambda_i\right),\quad \ell(i) \defn \mu_{i\Sigma},
\end{align*}
where $V_{h}(i)$ is a solution to the Poisson equation for $h(i)$. We now state the main result of this section, which provides an expression for the MGF of each element of $\widetilde{\vq}_\parallel=\left(\tq_\Sigma/n\right)\vone$.

\begin{theorem}\label{thm:jsq-heavy-traffic-distribution-q_avg}
    Consider a JSQ system with Markov-modulated arrival and service rates, as described in \Cref{sec:model}, and assume that the SSC load condition (\cref{def:condition_lambda_i-mu_i}) is satisfied.
    
    Let $V_{h}(i)$, $V_{k}(i)$, and $V_{\ell}(i)$ be a the solution to the Poisson equation for the functions $h(i)$, $k(i)$ and $\ell(i)$, respectively. Further, assume $\E{|V_{h}(i)|^{1+\eta}}<\infty$ for some $\eta>0$, $\E{|V_{k}(i)|^{1+\zeta}}<\infty$ for some $\zeta>0$, and $\E{|V_{\ell}(i)|^{1+\beta}}<\infty$ for some $\beta>0$. Then
    \begin{align*}
    \E{e^{-s\epsilon \tq_\Sigma/n}} &= \frac{1}{1+\frac{s}{n}\left(1+\frac{k^*}{\mu_\Sigma}\right)} + O\left(\epsilon^{1 - \frac{1}{1+\min\{1,\eta\}}}\right),
\end{align*}
where $k^*\defn \lim_{\epsilon\downarrow 0} \E{k(i)}$. 
\end{theorem}

%The conditions and the result in \cref{thm:jsq-heavy-traffic-distribution-q_avg} are similar to \cref{thm:jsq-vq-convergence}. The key difference is that \cref{thm:jsq-heavy-traffic-distribution-q_avg} provides the asymptotic MGF of the scaled \textit{average} queue length, $\epsilon\tq_\Sigma/n$, whereas \cref{thm:jsq-vq-convergence} deals with the scaled \textit{individual} queue lengths, $\epsilon \tq_j$ for each $j\in\{1,\ldots,n\}$. Hence, the latter result is more powerful. We use \cref{thm:jsq-heavy-traffic-distribution-q_avg} as a step in the proof of \cref{thm:jsq-vq-convergence}, and the results are connected by SSC.

The foundations of the proof of \cref{thm:jsq-heavy-traffic-distribution-q_avg} are similar to the single-server queue (see \cref{thm:ssq-heavy-traffic-distribution}). Indeed, we can pretend that $q_\Sigma/n$ is a single-server queue, and roughly, the same steps of the proof will follow. This is a consequence of SSC, which intuitively says that the JSQ system behaves as a single-server queue for small $\epsilon$. However, this is only an approximation. We need to carefully track the error terms that arise because for positive $\epsilon$, SSC is an approximation, and it is not an exact result. We now provide the proof.

\begin{proof}{Proof of \cref{thm:jsq-heavy-traffic-distribution-q_avg}.}
For ease of exposition, we prove the result for the total queue length. Note that all the steps in the proof are valid for any $s>0$. Hence, the result for the average queue length $q_\Sigma/n$ can be obtained by replacing $s$ by $s/n$. 
We start by computing the drift of $\varphi_s(\vq,i) = e^{-s\epsilon q_\Sigma}$:
\begin{align*}
    \Delta \varphi_s (\vq,i) &= \lambda_i \left(e^{-s\epsilon (q_\Sigma+1)}-e^{-s\epsilon q_\Sigma} \right) + \sum_{j=1}^n \ind{q_j>0} \mu_{ij} \left(e^{-s\epsilon (q_\Sigma-1)}-e^{-s\epsilon q_\Sigma} \right) \\
    &= (e^{-s\epsilon} -1) \left( e^{-s\epsilon q_\Sigma}(\lambda_i - \mu_{i\Sigma} e^{s\epsilon}) + e^{s\epsilon} e^{-s\epsilon q_\Sigma} \sum_{j=1}^n \mu_{ij} \ind{q_j=0} \right),
\end{align*}
where we rearranged terms using that $\ind{q_j>0}=1-\ind{q_j=0}$. 

Notice that $0<\E{\varphi_s(\vq,i)}\leq 1<\infty$. Then, we can set its drift to zero. 
Taking expectation under steady state and rearranging terms, we obtain:
\begin{align}\label{jsq-eq:service-mgf-to-solve}
    \E{e^{-s\epsilon \tq_\Sigma}(\mu_{i \Sigma} e^{s\epsilon} - \lambda_i)} =  e^{s\epsilon} \E{e^{-s\epsilon \tq_\Sigma} \sum_{j=1}^n \mu_{ij} \ind{\tq_j=0}}.
\end{align}

Setting $s=0$ in \eqref{jsq-eq:service-mgf-to-solve}, we obtain:
\begin{align}\label{jsq-eq:service-prob-empty-queue}
    \mu_\Sigma\epsilon = \E{\sum_{j=1}^n \mu_{ij}\ind{\tq_j=0}}.
\end{align}
In the single-server case, setting $s=0$ was sufficient to compute the right-hand side of the drift equation \eqref{eq:ssq-service-mgf-to-solve}. To do so, we used that $e^{s\epsilon \tq}\ind{\tq=0}=\ind{\tq=0}$. In this case, the indicator is for the $j\tth$ queue being empty while the exponential function is for the total queue length. Notice that if all the queue lengths are equal, then $\tq_j=0$ implies $\tq_\Sigma=0$. However, the queue lengths are not exactly equal to each other. Hence, we need to use SSC to solve for the right-hand side of \eqref{jsq-eq:service-mgf-to-solve}. We delay this step until the end of the proof.

Solving for the left-hand side of \eqref{jsq-eq:service-mgf-to-solve} requires using the Poisson equation, as the expectation involves the product of terms involving the queue lengths and the Markov-modulated parameters. This process is similar to the computation for the single-server queue.

Define
\begin{align*}
\cT \defn e^{-s\epsilon \tq_\Sigma}\left(\mu_{i\Sigma} e^{s\epsilon}-\lambda_i\right).
\end{align*}

Applying Taylor expansion on $e^{s\epsilon}$ we obtain
\begin{align*}
    \cT &= e^{-s\epsilon \tq_\Sigma} \left( \mu_{i\Sigma} e^{s\epsilon} -\lambda_i \right) 
    = e^{-s\epsilon \tq_\Sigma} \left(\mu_{i\Sigma}-\lambda_i\right) + s\epsilon e^{-s\epsilon \tq_\Sigma} \mu_{i \Sigma} + O(\epsilon^2),
\end{align*}
where the $O(\epsilon^2)$ arises because $0<e^{-s\epsilon \tq_\Sigma}\leq 1$ and $\mu_{i\Sigma}$ is bounded by assumption. Taking expectation with respect to the stationary distribution, we obtain
\begin{align}\label{jsq-eq:exp-T1}
    \E{\cT} = \E{e^{-s\epsilon\tq_\Sigma}(\mu_{i\Sigma}-\lambda_i)} + s\epsilon \E{e^{-s\epsilon \tq_\Sigma}\mu_{i\Sigma}} + O(\epsilon^2).
\end{align}
We now bound each of the expectations in \eqref{jsq-eq:exp-T1}. 
For the first term, we use \cref{thm:jsq-poisson-lemma} with $f(i)= h(i) =\mu_{i\Sigma}-\lambda_i$ and obtain:
\begin{align}
     \E{e^{-s \epsilon \tq_\Sigma} (\mu_{i \Sigma}-\lambda_i)}
    &=\E{e^{-s\epsilon \tq_\Sigma}} (\mu_\Sigma-\lambda) + (e^{-s\epsilon}-1) \E{e^{-s\epsilon \tq_\Sigma} V_{h}(i)(\lambda_i-\mu_{i\Sigma})} + O\left(\epsilon^{2-\frac{1}{1+\eta}}\right) \nonumber \\
    &\stackrel{(*)}{=} \epsilon \mu_\Sigma \E{e^{-s\epsilon \tq_\Sigma}} + s\epsilon \E{e^{-s\epsilon \tq_\Sigma} V_{h}(i)(\mu_{i\Sigma}-\lambda_i)} + O\left(\epsilon^{2-\frac{1}{1+\eta}}\right), \label{jsq-eq:exp-T1-Poisson-h}
\end{align}
where 
$(*)$ holds because $\epsilon\defn 1-\lambda/\mu_\Sigma$, expanding $e^{-s\epsilon}$ up to second order, and noticing that $e^{-s\epsilon \tq_\Sigma}$ and $\lambda_i-\mu_{i\Sigma}$ are bounded, and $V_h(i)$ has bounded expectation. Then, higher order terms are absorbed in $O\left(\epsilon^{2-\frac{1}{1+\eta}}\right)$. 

For the second term in \eqref{jsq-eq:exp-T1-Poisson-h}, we use \cref{thm:jsq-poisson-lemma} with $f(i)=k(i)=V_{h}(i)(\mu_{i\Sigma}-\lambda_i)$ and obtain
\begin{align}
    & \E{e^{-s\epsilon \tq_\Sigma} V_{h}(i)(\mu_{i\Sigma}-\lambda_i)} = \E{e^{-s\epsilon \tq_\Sigma}k(i)} \nonumber \\
    &= \E{e^{-s\epsilon \tq_\Sigma}}\E{k(i)} + (e^{-s\epsilon}-1)\E{e^{-s\epsilon\tq_\Sigma} V_{k}(i)(\lambda_i-\mu_{i\Sigma}) } + O\left(\epsilon^{2-\frac{1}{1+\zeta}}\right) \nonumber \\
    &= \E{e^{-s\epsilon \tq_\Sigma}}\E{k(i)} + O(\epsilon), \label{jsq-eq:exp-T1-Poisson-k} 
\end{align}
where \eqref{jsq-eq:exp-T1-Poisson-k} follows by and argument similar to \eqref{jsq-eq:exp-T1-Poisson-h}. Putting \eqref{jsq-eq:exp-T1-Poisson-h} and \eqref{jsq-eq:exp-T1-Poisson-k} together, yields
\begin{align} \label{jsq-eq:exp-T1-first-term}
     \E{e^{-s \epsilon \tq_\Sigma} (\mu_{i \Sigma}-\lambda_i)} = \epsilon \mu_\Sigma \E{e^{-s\epsilon \tq_\Sigma}} \left(1 + s\frac{\E{k(i)}}{\mu_\Sigma} \right) + O\left(\epsilon^{2-\frac{1}{1+\eta}}\right).
\end{align}

To compute the second term of \eqref{jsq-eq:exp-T1} we use \cref{thm:jsq-poisson-lemma} with $f(i)=\ell(i)=\mu_{i\Sigma}$ and obtain
\begin{align}
    \E{e^{-s\epsilon \tq_\Sigma}\mu_{i\Sigma}}&= \mu_\Sigma \E{e^{-s\epsilon \tq_\Sigma}} + (e^{-s\epsilon}-1)\E{e^{-s\epsilon \tq_\Sigma}V_{\ell}(i)(\lambda_i-\mu_{i\Sigma})} + O\left(\epsilon^{2-\frac{1}{1+\beta}}\right) \nonumber \\
    &= \mu_\Sigma\E{e^{-s\epsilon \tq_\Sigma}} + O(\epsilon), \label{jsq-eq:exp-T1-Poisson-l} 
\end{align}
where \eqref{jsq-eq:exp-T1-Poisson-l} again uses a similar argument to \eqref{jsq-eq:exp-T1-Poisson-h} and \eqref{jsq-eq:exp-T1-Poisson-k}. Therefore, using \eqref{jsq-eq:exp-T1-first-term} and \eqref{jsq-eq:exp-T1-Poisson-l} in \eqref{jsq-eq:exp-T1}, we obtain
\begin{align*}
    \E{e^{-s\epsilon \tq_\Sigma}(e^{s\epsilon}\mu_{i\Sigma} - \lambda_i) } = \epsilon \mu_\Sigma \E{e^{-s\epsilon \tq_\Sigma}} \Big(1 + s\Bigg(1 + \frac{\E{k(i)}}{\mu_\Sigma} \Bigg)\Big) + O\left(\epsilon^{2-\frac{1}{1+\eta}}\right).
\end{align*}

Plugging back to \eqref{jsq-eq:service-mgf-to-solve} yields
\begin{align}\label{eq:mgf-before-ssc}
    \mu_\Sigma\epsilon \E{e^{-s\epsilon \tq_\Sigma}}\Big(1 + s\Bigg(1 + \frac{\E{k(i)}}{\mu_\Sigma} \Bigg)\Big)+ O(\epsilon^{2-\frac{1}{1+\eta}}) = e^{s\epsilon} \E{e^{-s\epsilon \tq_\Sigma} \sum_{j=1}^n \mu_{ij} \ind{q_j=0}}.
\end{align}

The final step is to compute the right-hand side of \eqref{eq:mgf-before-ssc}. When we computed the right-hand side term for the single-server queue, we used the property $e^{-s\epsilon\tq}\ind{\tq=0}=\ind{\tq=0}$ (see the argument between \cref{{eq:ssq-service-mgf-to-solve}} and \cref{eq:ssq-service-prob-empty-queue}). In this case, we cannot use the same property directly because $\tq_j=0$ for some $j$ does not imply $\tq_\Sigma=0$. However, if all the queue lengths were equal, this property would hold. Hence, we use SSC to compute the right-hand side of \eqref{eq:mgf-before-ssc}. 
Specifically, we have
\begin{align*}
    e^{-s\epsilon \tq_\Sigma} \sum_{j=1}^n \mu_{ij} \ind{\tq_j=0} 
    &\stackrel{(a)}{=} e^{-s\epsilon \tq_\Sigma} \sum_{j=1}^n \mu_{ij} \ind{\tq_j=0} e^{s\epsilon n \tq_j} \\
    &\stackrel{(b)}{=} \sum_{j=1}^n \mu_{ij} \ind{\tq_j=0}  e^{s\epsilon n (\tq_j - \tq_\Sigma/n)}\\
    &\stackrel{(c)}{=} \sum_{j=1}^n \mu_{ij} \ind{\tq_j=0} e^{s\epsilon n \tq_{\perp j}} \\
    &\stackrel{(d)}{=} \sum_{j=1}^n \mu_{ij} \ind{\tq_j=0} (e^{s\epsilon n \tq_{\perp j}} - 1) + \sum_{j=1}^n \mu_{ij}\ind{\tq_j=0},
\end{align*}
where $(a)$ holds because $\ind{\tq_j=0}=e^{a \tq_j}\ind{\tq_j=0}$ for any constant $a\in\bR$;
$(b)$ by putting together the exponential functions and reorganizing terms; 
$(c)$ by definition of $\tvq_\perp$ in \eqref{eq:def-qpar-qperp}; 
and $(d)$ by adding and subtracting $\sum_{j=1}^n \mu_{ij}\ind{\tq_j=0}$, which is the term that resembles the single-server queue proof.

Taking the expectation, we obtain
\begin{align}
    \E{e^{-s\epsilon \tq_\Sigma} \sum_{j=1}^n \mu_{ij} \ind{\tq_j=0} } 
    &= \E{\sum_{j=1}^n \mu_{ij} \ind{\tq_j=0} (e^{s\epsilon n \tq_{\perp j}}-1)} + \E{\sum_{j=1}^n \mu_{ij} \ind{\tq_j=0}} \nonumber \\
    &\stackrel{(*)}{=} \E{\sum_{j=1}^n \mu_{ij} \ind{\tq_j=0} (e^{s\epsilon n \tq_{\perp j}}-1)} + \mu_\Sigma \epsilon,
    \label{eq:mgf-ssc-step1}
\end{align}
where $(*)$ uses \eqref{jsq-eq:service-prob-empty-queue}. We bound the first term in \eqref{eq:mgf-ssc-step1} using Cauchy-Schwarz inequality and the exponential bound of \cref{prop:ssc}. The detailed computation is presented in Appendix \ref{sec:details-JSQ-thm-proof}. We obtain
\begin{align*}
    \E{e^{-s\epsilon \tq_\Sigma}\sum_{j=1}^n \mu_{ij}\ind{q_j=0}} = \mu_\Sigma \epsilon + O(\epsilon^{\frac{3}{2}}).
\end{align*}

Finally, using this bound in \eqref{eq:mgf-before-ssc}, we obtain
\begin{align*}
    \mu_\Sigma\epsilon \E{e^{-s\epsilon \tq_\Sigma}}\Big(1 + s\Big(1 + \frac{\E{k(i)}}{\mu_\Sigma} \Big)\Big)+ O(\epsilon^{2-\frac{1}{1+\eta}}) = e^{s\epsilon} \mu_\Sigma\epsilon + O(\epsilon^{\frac{3}{2}}).
\end{align*}

Therefore, reorganizing terms and canceling $\epsilon \mu_\Sigma$, we obtain
\begin{align*}
    \E{e^{-s\epsilon \tq_\Sigma}} &= \dfrac{e^{s\epsilon}}{1 + s\Big(1 + \frac{\E{k(i)}}{\mu_\Sigma} \Big)} + O(\sqrt{\epsilon}) + O(\epsilon^{1 - \frac{1}{1+\eta}}) \\
    &\stackrel{(*)}{=} \dfrac{1}{1 + s\Big(1 + \frac{\E{k(i)}}{\mu_\Sigma} \Big)}+ O(\epsilon^{1 - \frac{1}{1+\min\{1,\eta\}}}),
\end{align*}
where $(*)$ holds because $e^{s\epsilon}=1 + O(\epsilon)$ if we expand in Taylor series. 

The error term $O(\epsilon^{1 - \frac{1}{1+\min\{1,\eta\}}})$ comes from the smallest between $O(\sqrt{\epsilon})$ and $O(\epsilon^{1 - \frac{1}{1+\eta}})$, which represent the error terms arising from SSC and the use of the Poisson equation. These are the two approaches we use to deal with cross-terms, as discussed in \cref{sec:proof-overview}.

The last step is to argue that $\left|k^* - \E{k(i)}\right| = O(\epsilon)$, so the denominator has a $k^*$ instead of $\E{k(i)}$. The argument is equivalent to the single-server queue, so we omit it. 
\Halmos
\end{proof}

\subsubsection{Rate of Convergence}
\label{sec:roc}

The last step in our proof is to compute the error of approximation between the scaled individual queue lengths $\epsilon\tq_j$ for each $j$, and their projection on the SSC region $\epsilon \tq_{\parallel j}=\epsilon (\tq_\Sigma/n)$. We formalize this result in the next theorem and present the proof of the rate of convergence for small $\epsilon$.

\begin{theorem}\label{thm:jsq-ROC}
    Consider the JSQ system as in \cref{thm:jsq-heavy-traffic-distribution-q_avg}. Then, the Laplace transform of each individual scaled queue length $\epsilon\tq_j$ is
    \begin{align*}
        \E{e^{-s\epsilon \tq_j}} = \E{e^{-s\epsilon \tq_\Sigma/n}} + O\left(\epsilon \ln \frac{1}{\epsilon}\right), \qquad \forall j\in\{1,\ldots,n\}, \forall s>0.
    \end{align*}
\end{theorem}
Note that this result also shows that the vector of queue lengths $\vq$ converges to $\vq_{\parallel}$ at the same rate.

\begin{proof}{Proof of \cref{thm:jsq-ROC}.}
First, recall that $\tvq = (\frac{\tq_\Sigma}{n}) \vone + \tvq_\perp$,
or equivalently, that $\tq_j = \frac{\tq_\Sigma}{n} + \tvq_{\perp j}$ for each $j\in\{1,\ldots,n\}$. As a result, 
\begin{align*}
    e^{-s\epsilon \tq_j} - e^{-s\epsilon \tq_\Sigma/n} &= e^{-s\epsilon (\tq_\Sigma/n + \tq_{\perp j})} - e^{-s\epsilon q_\Sigma/n} 
    = e^{-s\epsilon \tq_\Sigma/n} (1-e^{-s\epsilon \tq_{\perp j}}).
\end{align*}

Our goal is now to prove that the expected value of this quantity is bounded by $O(\epsilon \ln \frac{1}{\epsilon})$. 
To do so, we make use of our tail bound \eqref{eq:qperp-bound-tail-prob} in \cref{prop:ssc}. Specifically,
\begin{align*}
    P(\|\tvq_\perp\| > b+vp) \leq a^{p+1},
\end{align*}
with $b,v>0$ and $0<a<1$. In particular, from the proof of \cref{prop:ssc} in Appendix~\ref{app:ssc-proof}, we see that $b$ and $v$ are uniformly bounded constants.

Define $c= b+v$ and note that $c>0$. Additionally, the constants $a$ and $c$ are independent of $\epsilon$ Then, 
for all integers $p > 0$,
\begin{align*}
    P(\|\tvq_\perp\| > c p) \le a^p.
\end{align*}

Set $p = \lceil\log_{a} \epsilon \rceil$ and note that $p > 0$ because $0 < a < 1$ and $0 < \epsilon < 1$.
Thus, we find that
\begin{align} \label{eq:jsq-ssc-tail-prob}
    P(\|\tvq_\perp\| > c \lceil\log_{a} \epsilon \rceil) \le \epsilon, \quad
    \& \quad 
    P(\tq_{\perp j} > c \lceil \log_{a} \epsilon \rceil) \le \epsilon.
\end{align}

Next, we condition on the event $[q_{\perp j} > c \lceil \log_{a} \epsilon\rceil ]$:
\begin{align*}
     \E{e^{-s\epsilon \tq_\Sigma/n} (1-e^{-s\epsilon \tq_{\perp j}})}
    &\le \E{ \left. e^{-s\epsilon \tq_\Sigma/n} (1-e^{-s\epsilon \tq_{\perp j}}) \right| \tq_{\perp j} > c \lceil \log_{a} \epsilon \rceil}
    P(\tq_{\perp j} > c \lceil\log_{a} \epsilon \rceil) \\
    &\quad + \E{ \left. e^{-s\epsilon \tq_\Sigma/n} (1-e^{-s\epsilon \tq_{\perp j}}) \right| \tq_{\perp j} \le c \lceil\log_{a} \epsilon \rceil}
    P(\tq_{\perp j} \le c \lceil\log_{a} \epsilon \rceil) \\
    &\stackrel{(a)}{\leq} 1 \cdot \epsilon + \E{e^{-s\epsilon q_\Sigma/n} (1-e^{-s\epsilon c \lceil\log_{a} \epsilon \rceil})} \\
    &\stackrel{(b)}{\leq} \epsilon + (1-e^{-s\epsilon c \lceil\log_{a} \epsilon \rceil})
    \le \epsilon + s\epsilon c \lceil\log_{a} \epsilon \rceil
    = O(\epsilon \ln \frac{1}{\epsilon}),
\end{align*}
where 
$(a)$ holds by upper bounding the first expectation by one under the condition, the first probability by $\epsilon$ according to \eqref{eq:jsq-ssc-tail-prob}. To upper bound the second expectation, we use the condition on $\tq_{\perp j}$ and that probabilities are always smaller than one.
Next, $(b)$ holds because $e^{-s\epsilon \tq_\Sigma/n}\leq 1$, and the remaining inequalities hold by rearranging terms. 
\Halmos
\end{proof}

\section{Empirical Validation}
\label{sec:empirical}

We have characterized the asymptotic mean queue length behavior of the Markov-modulated Join-the-Shortest-Queue (JSQ) system.
We now empirically validate and illustrate our theoretical results via simulation, and compare the measured response time against our predictions.

We focus on a setting with 3 servers and a 3-state Markov chain, with a single free parameter $r$:
\begin{align}
    \label{eq:main-setting}
    \vlambda = [3r, 6r, 9r],
    [\mu_{ij}] = \begin{pmatrix}
        0.5 & 0.5 & 1 \\
        1 & 2.5 & 2 \\
        5 & 3 & 2.5
    \end{pmatrix},
    [\alpha_{i i'}] = \begin{pmatrix}
        0 & 0.1 & 0 \\
        0 & 0 & 0.1 \\
        0.1 & 0 & 0
    \end{pmatrix}.
\end{align}

For this chain, the load $\rho$ is equal to the adjustable parameter $r$.
As for the SSC load condition specified in \cref{def:condition_lambda_i-mu_i}, note that $\mu_{i \Sigma} - n\mu_{i \min} = [0.5, 2.5, 3]$, so the condition holds for all $r \ge \frac{5}{12}$, and consequently our heavy traffic result \cref{thm:jsq-vq-convergence} applies to this setting.
In this setting, $k^* = \frac{35}{6}$.

\subsection{Convergence of mean and distribution in heavy traffic}

\begin{figure}
\FIGURE{
{\subfloat[Scaled mean queue length $(1-\rho)\E{q_j}$ converging to heavy-traffic mean.]{\includegraphics[width=0.49\textwidth]{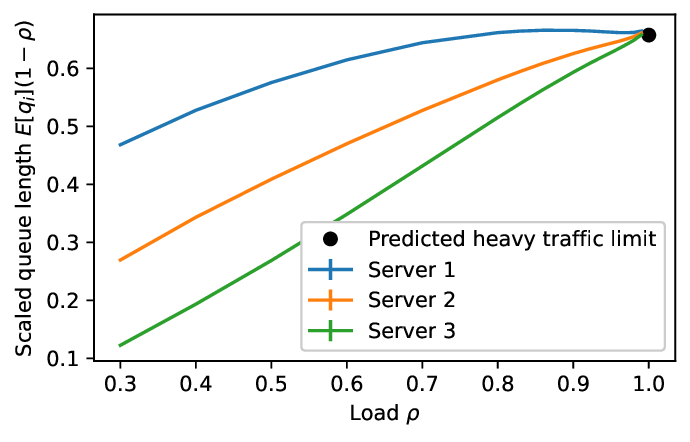}}}
\hfill
{\subfloat[PMF of individual queue length with load $\rho=0.98$, against limiting exponential distribution. Vertical axis is logarithmic.]{\includegraphics[width=0.49\textwidth]{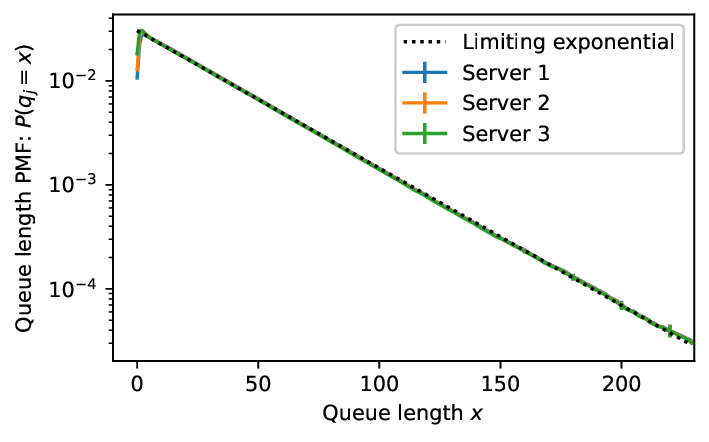}}}
}
% Caption
{Convergence of each server's queue length to heavy traffic.}
% Notes
{$10^9$ arrivals per data point, as 10 runs of $10^8$. 95\% CI shown, or too small to be visible at this scale. SSC assumption satisfied: setting \eqref{eq:main-setting}.\label{fig:convergence-by-load}}
\end{figure}

In \cref{fig:convergence-by-load}(a), we simulate the scaled mean queue length $\E{q_j}(1-\rho)$ for each server $j$, as a function of load $\rho$ in the setting given in \eqref{eq:main-setting}. The black dot shows the heavy traffic limit proven in \cref{thm:jsq-vq-convergence}, 
namely $\frac{1}{n}(1+k^*/\mu_\Sigma)$, which evaluates to $\frac{71}{108}$ in this setting. The mean queue lengths of all three servers converge smoothly to the predicted heavy traffic limit, as we proved would occur.

\cref{fig:convergence-by-load}(b) shows the probability mass functions $\Prob{q_j = x}$ each server $j$, at a relatively heavy load of $\rho = 0.98$. The distribution is almost identical to an exponential distribution with mean $\frac{1}{1-\rho}\frac{1}{n}(1+k^*/\mu_\Sigma)$ (in this case, $\frac{1775}{54}$), as predicted by the heavy traffic limit proven in \cref{thm:jsq-vq-convergence}.
All servers $j$ have nearly indistinguishable distributions.
Queue lengths very close to 0 are the
slowest part of the distribution to converge to the limiting exponential, while all larger queue lengths are extremely close to their limiting values even at $\rho = 0.98$.

For each queue $j$, the scaled mean queue length $\epsilon \E{q_j}$ converges to $\frac{1}{n}(1+k^*/\mu_\Sigma) = \frac{71}{108}$.

\subsection{State space collapse and mean under varying \texorpdfstring{$\alpha$}{alpha}}
\label{sec:varying-alpha-with-ssc-condition}

\begin{figure}
\FIGURE{
\subfloat[Mean SSC error per server $\E{|q_{\perp j}|}=\E{|q_j - q_\Sigma/n|}$. Horizontal axis is logarithmic, and vertical is linear.]{\includegraphics[width=0.45\textwidth]{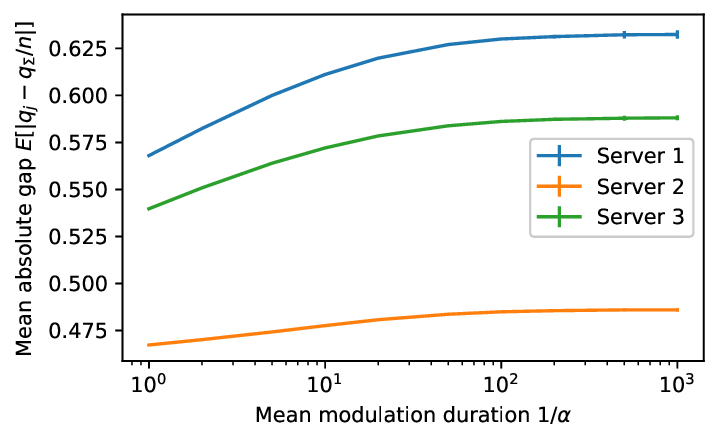}}
\hfill
\subfloat[Mean queue length $\E{q_j}$ and heavy traffic limit. Both axes are logarithmic.]{\includegraphics[width=0.45\textwidth]{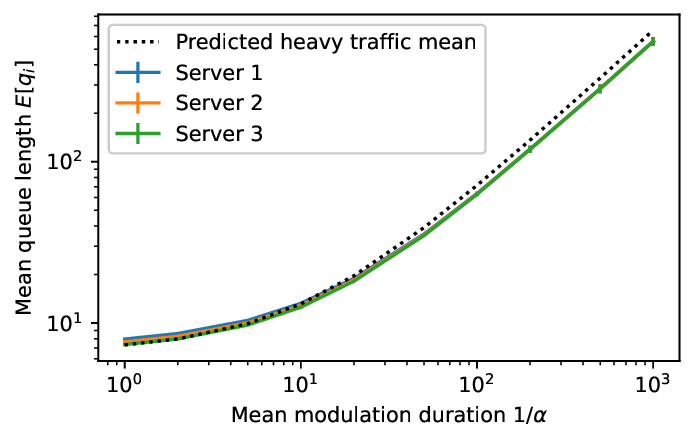}}
}
% Caption
{State space collapse quality with respect to mean modulation time $\frac{1}{\alpha}$.}
% Notes
{$10^9$ arrivals per data point, as 10 runs of $10^8$. 95\% CI shown. SSC assumption satisfied: setting \eqref{eq:adjustable-alpha-setting}. Load $\rho=0.95$.\label{fig:ssc-alpha-0}}
\end{figure}

In this section, we allow the modulation rate $\alpha$ to vary, and examine its effect on the tightness of the SSC and the convergence to the predicted mean queue distribution.

We are using a variant of the setting in \eqref{eq:main-setting}, where both load and modulation rate $\alpha$ are adjustable parameters:
\begin{align}
    \label{eq:adjustable-alpha-setting}
    \vlambda = [3r, 6r, 9r],
    [\mu_{ij}] = \begin{pmatrix}
        0.5 & 0.5 & 1 \\
        1 & 2.5 & 2 \\
        5 & 3 & 2.5
    \end{pmatrix},
    [\alpha_{i i'}] = \begin{pmatrix}
        0 & \alpha & 0 \\
        0 & 0 & \alpha \\
        \alpha & 0 & 0
    \end{pmatrix}.
\end{align}

Again, the load $\rho$ is equal to the parameter $r$, and the SSC load condition holds for all $r \ge \frac{5}{12}$. Consequently, our heavy traffic result \cref{thm:jsq-vq-convergence} applies to this setting.
In this setting, $k^* = \frac{7}{12}\frac{1}{\alpha}$.

In \cref{fig:ssc-alpha-0}(a), we show the simulated mean absolute gap between the queue length of each server and the average over all servers, at load $\rho = 0.95$, across a range of mean modulation durations $\frac{1}{\alpha}$. Our SSC result, \cref{prop:ssc}, implies that this gap should be very close to 0, and should not grow, even as the mean modulation duration $\frac{1}{\alpha}$ grows without bound.
We see exactly this behavior, with mean absolute gap remaining bounded below a single job, essentially as small a gap as is achievable, regardless of modulation rate $\alpha$.

In \cref{fig:ssc-alpha-0}(b), we simulate the effect on mean response time of changing mean modulation duration $\frac{1}{\alpha}$. Note that $k^*$, a key parameter in our mean queue length heavy traffic limit, is equal to $\frac{7}{12}\frac{1}{\alpha}$, so our heavy traffic result predicts that mean queue length will grow along a specific trajectory as modulation duration $\frac{1}{\alpha}$ grows, as we see replicated in simulation for each server. The approximation is close, but not exact, as expected at load $\rho = 0.95$.

\subsection{Varying \texorpdfstring{$\alpha$}{alpha} outside our SSC load condition}
\label{sec:empirical-outside-condition}

\begin{figure}
\FIGURE{
\subfloat[Mean SSC error per server $\E{|q_{\perp j}|}=\E{|q_j - q_\Sigma/n|}$.]{\includegraphics[width=0.49\textwidth]{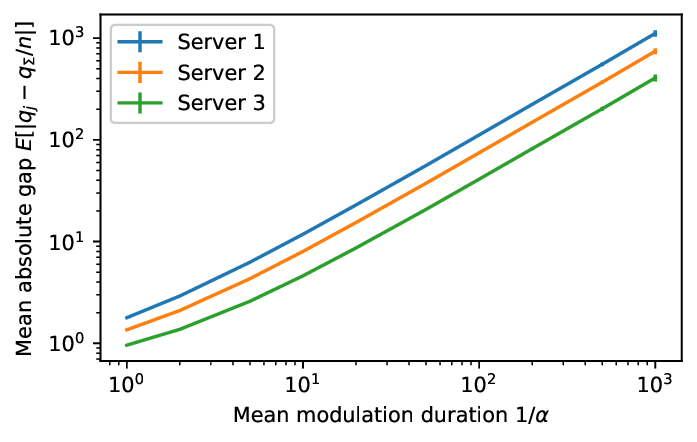}}
\hfill
\subfloat[Mean queue length $\E{q_j}$ and heavy-traffic limit.]{\includegraphics[width=0.49\textwidth]{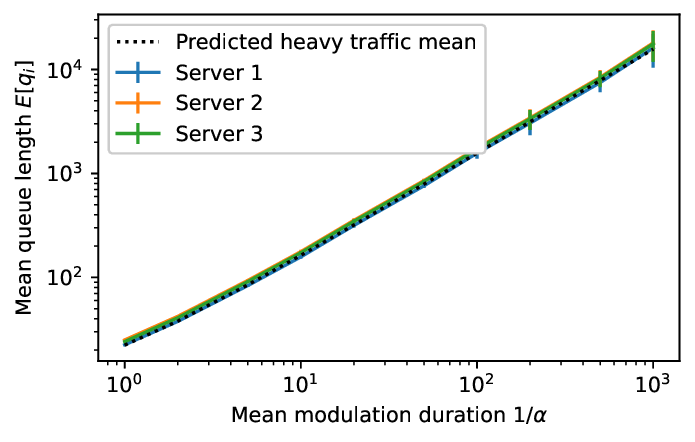}}
}
% Caption
{State space collapse quality with respect to mean modulation time $\frac{1}{\alpha}$ when the SSC load condition is not satisfied.}
% Note
{$10^9$ arrivals per data point, as 10 runs of $10^8$. 95\% CI shown. SSC assumption not satisfied: setting \eqref{eq:outside-ssc-setting}. Load $\rho=0.95$. Both axes logarithmic.\label{fig:ssc-alpha-1}}
\end{figure}

We now turn to a setting where our SSC load condition (\cref{def:condition_lambda_i-mu_i}) \emph{does not hold},
and as a consequence our heavy traffic results do not apply,
to examine the distinct behavior that arises in this setting.
\begin{align}
    \label{eq:outside-ssc-setting}
    \vlambda = [3r, 6r, 9r],
    [\mu_{ij}] = \begin{pmatrix}
        8 & 0.5 & 1 \\
        0.5 & 1 & 0.5 \\
        1.5 & 1.5 & 2.5
    \end{pmatrix},
    [\alpha_{i i'}] = \begin{pmatrix}
        0 & \alpha & 0 \\
        0 & 0 & \alpha \\
        \alpha & 0 & 0
    \end{pmatrix}.
\end{align}

Again, load $\rho$ is equal to the parameter $r$.
For SSC load condition specified in \cref{def:condition_lambda_i-mu_i}, note that $\mu_{i \Sigma} - n\mu_{i \min} = [6.5, 0.5, 1]$, so the condition fails in state $i=1$ for all $r < 1$, as $\lambda_1 < 3 < \mu_{\Sigma 1} - n \mu_{1,2} = 6.5$. As a consequence, our heavy traffic result \cref{thm:jsq-vq-convergence} does not apply to this setting.
In this setting, $k^* = \frac{169}{12} \frac{1}{\alpha}$.

In \cref{fig:ssc-alpha-1}(a), we see a very different behavior than in \cref{sec:varying-alpha-with-ssc-condition}, where the SSC load condition holds. Here, the mean absolute gap between different servers rises linearly with mean modulation duration $\frac{1}{\alpha}$, resulting in worse and worse SSC quality for larger and larger mean modulation duration. This example shows that when the SSC load condition fails, state space collapse can cease to hold entirely, or have completely different dependence on system parameters.

Nonetheless, in \cref{fig:ssc-alpha-1}(b), we see that despite the SSC quality worsening with larger mean modulation duration $\frac{1}{\alpha}$, mean queue length still follows the same predicted mean queue length curve as in \cref{sec:varying-alpha-with-ssc-condition}. In fact, the convergence to the heavy traffic prediction is even tighter in \cref{fig:ssc-alpha-1}(b). More research is needed to understand why the mean convergence holds in the absence of tight state space collapse, and how broadly this convergence holds.

\section{Conclusion}
\label{sec:conclusion}

We prove \cref{thm:jsq-vq-convergence},
the first heavy traffic convergence result for the infinite-state Markov-modulated Join-the-Shortest-Queue (JSQ) system.
Our proof relies on a sufficient condition for State Space Collapse (SSC) for JSQ, \cref{def:condition_lambda_i-mu_i},
and our SSC result is the first to have no dependence on the mixing time of the modulating chain.
We employ a novel variant of the transform method, using the Poisson equation for the modulating chain to obtain a key covariance result, \cref{thm:jsq-poisson-lemma}, which lies at the core of our proof.
Our empirical results in \cref{sec:empirical-outside-condition} show that SSC indeed becomes dependent on the mixing time of the modulating chain outside of our sufficient condition, but nonetheless, the JSQ system can maintain the same heavy traffic limiting behavior that we prove.
Future work is needed to understand this larger region of heavy-traffic convergence.
Another fruitful avenue to future work is to expand our JSQ result to the generalized switch.

%\THEEndNotes
% \begingroup \parindent 0pt \parskip 4ex
% \def\enotesize{\normalsize} 
% \theendnotes
% \endgroup

% Appendix here

\newpage

\begin{APPENDICES}
\section{Transform Method for an M/M/1 queue}
\label{app:transform-MM1}

In this section, we show that the steady-state distribution of the queue length in an $M/M/1$ queue with arrival rate $a$ and service rate $b$ is geometric with parameter $1-a/b$ (assuming $a<b$). We start by computing the drift of the function $\phi_s(q)=e^{-s q}$, with $s>0$:
\begin{align*}
    \Delta \phi_s(q) &= a e^{-s q}(e^{-s}-1) + b\ind{q>0} e^{-s q}(e^s-1) \\
    &= (a-e^s b)e^{-s q}(e^{-s}-1) +b\ind{q=0}e^s (e^{-s}-1),
\end{align*}
where the second equality results from the key property $e^{-sq} \ind{q=0}=\ind{q=0}$ and algebraic manipulations. Taking expectation with respect to the stationary distribution yields
\begin{align*}
    \E{e^{-s q}} = \frac{\Prob{q=0}}{1-e^{-s} a/b}.
\end{align*}
Setting $s=0$, one obtains that $\Prob{q=0}=1-a/b$. Hence,
\begin{align*}
    \E{e^{-s q}} = \frac{1-a/b}{1-e^{-s}(1-a/b)},
\end{align*}
which corresponds to the geometric distribution.

\section{Proof of SSC}
\label{app:ssc-proof}

The proof of SSC is standard in the drift-method literature. Specifically, we use a result first proved in \cite[Lemma 1]{Ery-Sri-2012-drift} for discrete-time systems. The bounds from \cite{Ery-Sri-2012-drift} were further tightened in \cite[Lemma 3]{Mag-Sri-2016-switch}, and later extended to continuous-time systems in \cite[Supplemental Material, Lemma 10]{Wang-Mag-Sri-2022-Bandwidth-Sharing}. We present this essential result below for completeness, and we use the notation from \cite[Lemma 1]{HL-Mag-2022-JSQ-Many-Server}.

\begin{lemma}\label{lem:weina}
    Let $\left\{X(t):t\in\bR_+ \right\}$ be a CTMC over a countable state space $\cX$, with transition rate matrix $G^X$. Suppose that it is irreducible, nonexplosive and positive recurrent, and it converges in distribution to a random variable $\widetilde{X}$ as $t\to\infty$. Consider a Lyapunov function $W:\cX\to\bR_+$ and suppose its drift satisfies the following conditions:
    \begin{enumerate}[label=(C\arabic*)]
        \item There exist constants $\gamma>0$ and $B>0$ such that $\Delta W(x)\leq -\gamma$ for any $x\in\cX$ with $W(x)>B$,
        \item $\nu_{\max}\defn \sup\left\{|W(x')-W(x)|: x,x'\in\cX \text{ and }G^X_{x,x'}>0 \right\}<\infty$,
        \item $\bar{G}\defn \sup\left\{-G^X_{x,x}: x\in\cX \right\}<\infty$.
    \end{enumerate}
    Then, for any nonnegative integer $p$, we have
    \begin{align}\label{eq:Weina-lemma-tail-bounds}
        P(W(\widetilde{X})>B+2\nu_{\max} p) \leq \left(\dfrac{G_{\max}\nu_{\max}}{G_{\max}\nu_{\max}+\gamma} \right)^{p+1}.
    \end{align}
    
    As a result, for any positive integer $r$, the $r\tth$ moment of $W(\widetilde{X})$ can be bounded as follows:
    \begin{align}\label{eq:Weinna-lemma-moments}
        \E{W(\widetilde{X})^r}\leq C_r,
    \end{align} 
    where
    \begin{align*}
            C_r &\defn (2B)^r + (4\nu_{max})^r\left(\dfrac{G_{max}\nu_{max}+\gamma}{\gamma} \right)^r r!
            \quad \& \quad
        G_{\max}&\defn \sup\left\{\sum_{x'\in\cX: W(x)<W(x')}G^X_{x,x'} :x\in\cX \right\}.
    \end{align*}
\end{lemma}

\cref{lem:weina} provides an upper bound on the moments of $W(\widetilde{X})$, but not on the moment generating function. In the next corollary, we compute such bound using \eqref{eq:Weinna-lemma-moments}. We present the proof in \cref{sec:proof-cor:ssc-expo}
\begin{corollary}\label{cor:ssc-expo}
    Let $\{X(t):t\in\bR_+\}$ be a CTMC with state space $\cX$ and $W:\cX\to\bR_+$ satisfy the conditions in \cref{lem:weina}. Let
    \begin{align*}
        \Theta \defn \dfrac{\gamma}{4\nu_{\max}(G_{\max}\nu_{\max}+\gamma)},\quad \& \quad 
        C_{\exp} \defn e^{2B\Theta} + \dfrac{\gamma}{\gamma - (4\Theta\nu_{\max})(G_{\max}\nu_{\max}+\gamma)}.
    \end{align*}    
    Then, for all $\theta\in\bR$ with $|\theta|<\Theta$, we have 
    $\E{e^{\theta W(\widetilde{X})}} \leq C_{\exp}$.
\end{corollary}

To prove \cref{prop:ssc}, we use the following definitions and lemmas. Given a vector $\vq\in\bZ_+^n$, define the following functions:
\begin{align}
    & \widehat{W}(i,\vq)\defn \left\|\vq \right\|^2 ,\quad \widehat{W}_{\parallel}(i,\vq)\defn \| \vq_\parallel \|^2 ,\quad W_{\perp}(i,\vq)\defn \left\| \vq_\perp \right\|  \label{eq:What-Wpar-Wperp-i}.
\end{align}

Observe
\begin{align*}
    \widehat{W}_{\parallel}(i,\vq)= \| \vq_\parallel \|^2 = \sum_{j'=1}^n \left(\frac{1}{n}\sum_{j=1}^n q_j\right)^2  = \frac{1}{n}\left(\sum_{j=1}^n q_j\right)^2.
\end{align*}

We now prove \cref{prop:ssc}, using the conditions from \cref{lem:weina} for $W(\vq)=W_\perp(\vq)=\|\vq_\perp\|$.

\begin{proof}{Proof of conditions (C2) and (C3)}
    For condition (C2), observe that the only values of $\vq'$ with positive rate are when there is an arrival, a departure, or a transition in the modulating Markov chain. The last case does not affect the test function $W_\perp(\vq)$. An arrival or departure changes one queue length by 1, so $\|\vq\|$ changes by at most 1, and $\|\vq_\parallel\|$ by at most $1/\sqrt{n}$. Hence, condition $(C2)$ is satisfied with $\nu_{\max} = 1+\frac{1}{\sqrt{n}}$.

    For condition $(C3)$, observe that 
        $\bar{G} =  \lambda_{\max} + n\mu_{\max} + \max_{i\in\cZ} \sum_{i'\neq i} \alpha_{ii'}$ 
    is finite. Hence, $(C3)$ is satisfied. 
    \Halmos
\end{proof}

Proving condition $(C1)$ is more involved, and we use the following lemma, which was proved in \cite[Lemma 7]{HL-Mag-2022-JSQ-Many-Server}.
\begin{lemma}\label{lemma:ssc-concavity}
    Given a vector $\vq\in\bZ_+^n$, consider the functions $\widehat{W}(i,\vq),\, \widehat{W}_\parallel(i,\vq)$ and $W_\perp(i,\vq)$ defined in \cref{eq:What-Wpar-Wperp-i}. Then,
    \begin{align*}%\label{eq:Wperp-concave}
        \Delta W_\perp(i,\vq) \leq 
        \dfrac{1}{2\left\|\vq_\perp\right\|}\left(\Delta \widehat{W}(i,\vq) - \Delta \widehat{W}_\parallel (i,\vq) \right).
    \end{align*}
\end{lemma}

\begin{proof}{Proof of (C1).}
     We use \Cref{lemma:ssc-concavity} and start by computing $\Delta \widehat{W}(i,\vq)$ and $\Delta \widehat{W}_\parallel(i,\vq)$. Let $j^*$ be the queue to which a new arrival is routed, and let $q_{\min}\defn \min_{j\in\{1,\ldots,n\}} q_j$ be its number of jobs. Then, 
    \begin{align}
        \Delta \widehat{W}(i,\vq) &= \lambda_i\left(\|\vq+\ve^{(j^*)}\|^2 - \|\vq\|^2 \right)  + \sum_{j=1}^n \mu_{ij}\ind{q_j>0}\left(\|\vq-\ve^{(j)}\|^2 - \|\vq\|^2 \right) \nonumber \\
        &= \lambda_i \left(2q_{\min} + 1 \right) + \left(\sum_{j=1}^n \mu_{ij} \ind{q_j>0}(-2q_j) \right) \nonumber \\
        &= \big(\lambda_i + \sum_{j=1}^n \ind{q_j>0}\mu_{ij} \big) + 2\lambda_i q_{\min} - 2\big( \sum_{j=1}^n \mu_{ij}q_j\big). \label{eq:ssc-DeltaV-1}
    \end{align}

    Let
    \begin{align*}
        \delta_i \defn \dfrac{1}{n}\left(1 - \dfrac{\lambda_i}{\mu_{i\Sigma}} \right),
    \end{align*}
    and define the hypothetical arrival rate vector in state $i$ as 
    \begin{align}\label{eq:def-lambda-prime}
        \vlambda_i' \defn \vmu_i - \delta_i \mu_{i\Sigma}\vone.
    \end{align}
    Notice that $\lambda_i =\lambda'_{i\Sigma}$, with $\lambda'_{i\Sigma}\defn \sum_{j=1}^n \lambda'_{ij}$.  Then, adding and subtracting $2\langle \vlambda_i',\vq\rangle$ to \eqref{eq:ssc-DeltaV-1}, we obtain
    \begin{align}\label{eq:ssc-DeltaV-2}
        \Delta \widehat{W}(i,\vq) &= \big(\lambda_i + \sum_{j=1}^n \ind{q_j>0}\mu_{ij} \big) + 2 \big( \lambda_i q_{\min} - \sum_{j=1}^n \lambda_{ij}' q_j \big) - 2\big(\sum_{j=1}^n \mu_{ij}q_j \big) + 2\big(\sum_{j=1}^n \lambda'_{ij}q_j \big). 
    \end{align}
    The first term is bounded by $\lambda_i+\mu_{i\Sigma}$, which is also bounded. For the second term, we obtain
    \begin{align}
        \lambda_i q_{\min} - \sum_{j=1}^n \lambda_{ij}' q_j &\stackrel{(a)}{=} (\lambda_i - \lambda'_{i\Sigma}) q_{\min} + \sum_{j=1}^n \lambda_{ij}' (q_{\min}-q_j) \nonumber \\
        &\stackrel{(b)}{=} \sum_{j=1}^n \lambda_{ij}' (q_{\min}-q_j) \nonumber \\
        &\stackrel{(c)}{\leq} - \lambda_{i\min}' \sum_{j=1}^n (q_j-q_{\min}), \label{eq:ssc-DeltaV-q_perp0}
    \end{align}
    where $(a)$ holds by adding and substracting $\lambda_{i\Sigma}'q_{\min}$ and rearranging terms; $(b)$ because $\lambda_{i\Sigma}'=\lambda_i$ by definition of $\vlambda_i'$; and $(c)$ because $q_j\geq q_{\min}$ for all $j$, and letting $\lambda_{i\min}' \defn \min_{j\in\{1,\ldots,n\}} \lambda_{ij}'>0$. Notice that $\lambda_{i \min}'>0$ is equivalent to the SSC load condition:
    \begin{align*}
        \lambda_i > \mu_{i\Sigma} - n\mu_{i \min}.
    \end{align*}
    We now bound \eqref{eq:ssc-DeltaV-q_perp0}. We have
    \begin{align} \label{eq:ssc-DeltaV-q_perp1}
        \sum_{j=1}^n (q_j-q_{\min}) &= \|\vq-q_{\min}\vone\|_1
        \stackrel{(a)}{\geq} \|\vq-q_{\min}\vone\|
        \stackrel{(b)}{\geq} \|\vq-\vq_\parallel\|
        \stackrel{(c)}{=} \|\vq_\perp\|,
    \end{align}
    where $(a)$ holds because the Euclidean (L2) norm upper bounds the L1 norm; $(b)$ because $\vq_\parallel$ is the projection of $\vq$ on $\vone$, so it minimizes the function $\psi(x)=\|\vq-x\vone\|$; and $(c)$ by definition of $\|\vq_\perp\|$. 
    
    Using \eqref{eq:ssc-DeltaV-q_perp1} in \eqref{eq:ssc-DeltaV-q_perp0}, we obtain
    \begin{align}\label{eq:ssc-DeltaV-q_perp2}
        \lambda_i q_{\min} - \sum_{j=1}^n \lambda_{ij}' q_j \leq -\lambda'_{i\min} \|\vq_\perp\|.
    \end{align}
    
    Then, using \eqref{eq:ssc-DeltaV-q_perp2} in \eqref{eq:ssc-DeltaV-2} and by definition of $\delta_i$, we obtain
    \begin{align}\label{eq:ssc-DeltaV-ready}
        \Delta \widehat{W}(i,\vq) &= \big(\lambda_i + \sum_{j=1}^n \ind{q_j>0}\mu_{ij} \big) - 2\lambda'_{i\min}\|\vq_\perp\| - 2\delta_i \mu_{i\Sigma}\big(\sum_{j=1}^n q_j\big).
    \end{align}

    Now we compute $\Delta \widehat{W}_\parallel(i,\vq)$. By definition of drift, we obtain
    \begin{align}
        \Delta \widehat{W}_\parallel(i,\vq) &= \frac{\lambda_i}{n}\Big( (\sum_{j=1}^n q_j + 1)^2 - (\sum_{j=1}^n q_j)^2\Big) + \frac{1}{n}\Big(\sum_{j=1}^n \ind{q_j>0}\mu_{ij}\Big)\Big( (\sum_{j=1}^n q_j - 1)^2 - (\sum_{j=1}^n q_j)^2\Big) \nonumber \\
        &= \dfrac{\lambda_i}{n} (2q_\Sigma + 1) + \frac{1}{n}\Big(\sum_{j=1}^n \ind{q_j>0}\mu_{ij}\Big)(-2q_\Sigma + 1) \nonumber \\
        &= \frac{1}{n}\Big(\lambda_i + \sum_{j=1}^n \ind{q_j>0}\mu_{ij} \Big) + 2\lambda_i \Big(\frac{q_\Sigma}{n}\Big) - 2\Big( \sum_{j=1}^n \ind{q_j>0} \mu_{ij}\Big) \Big(\frac{q_\Sigma}{n}\Big). \label{eq:ssc-DeltaV_par-ready}
    \end{align}

    Therefore, taking the difference of the drifts \eqref{eq:ssc-DeltaV-ready} and \eqref{eq:ssc-DeltaV_par-ready} and using that $\ind{q_j>0}\leq 1$ and $1-1/n\geq 0$, we obtain
    \begin{align*}
        \Delta \widehat{W}(i,\vq) - \Delta \widehat{W}_\parallel (i,\vq) &\leq \left(1-\frac{1}{n}\right) \big(\lambda_i + \mu_{i\Sigma} \big) - 2\lambda_{i\min}'\|\vq_\perp\| \\
        &\; - 2\delta_i \mu_{i\Sigma}q_\Sigma - 2\lambda_i \Big(\frac{q_\Sigma}{n}\Big) + 2\Big( \sum_{j=1}^n \ind{q_j>0} \mu_{ij}\Big) \Big(\frac{q_\Sigma}{n}\Big) .
    \end{align*}

    For the last term, observe
    \begin{align*}
        \sum_{j=1}^n \ind{q_j>0} \mu_{ij} &\leq \sum_{j=1}^n \mu_{ij}
        = \sum_{j=1}^n (\lambda_{ij}' + \delta_i \mu_{i\Sigma})
        = \lambda_{i\Sigma}' + n\delta_i \mu_{i\Sigma} = \lambda_i + n\delta_i \mu_{i\Sigma}.
    \end{align*}
    
    Substituting above, we obtain
    \begin{align*}
        \Delta \widehat{W}(i,\vq) - \Delta V_\parallel (i,\vq) &\leq \left(1-\frac{1}{n}\right) \big(\lambda_i + \mu_{i\Sigma} \big) - 2\lambda_{i\min}'\|\vq_\perp\|.
    \end{align*}

    Finally, we use the last expression in \Cref{lemma:ssc-concavity} to obtain
    \begin{align*}
        \Delta W_\perp(i,\vq) &\leq \dfrac{1}{2\|\vq_\perp\|}\left( \left(1-\frac{1}{n}\right) \big(\lambda_i + \mu_{i\Sigma} \big) - 2\lambda_{i\min}'\|\vq_\perp\| \right) \\
        &= \left(1-\frac{1}{n}\right) \left(\frac{\lambda_i + \mu_{i\Sigma}}{2\|\vq_\perp\|} \right)- \lambda_{i\min}'.
    \end{align*}
    Hence, condition (C1) is satisfied with
    \begin{align*}
        \gamma = \frac{1}{2} \min_{i\in\cZ} \lambda_{i \min}',\quad B = \left(1-\frac{1}{n}\right)\max_{i\in\cZ} \left(\frac{\lambda_i + \mu_{i\Sigma}}{\lambda_{i \min}'}\right).
    \end{align*}
    This completes the proof of SSC. \Halmos
\end{proof}

\subsection{SSC and exponential bounds}
\label{sec:proof-cor:ssc-expo}

\begin{proof}{Proof of \Cref{cor:ssc-expo}}
    Considering the Taylor expansion of $e^{\theta W(\widetilde{X})}$ for $\theta\in\bR$, we have
    \begin{align}\label{eq:exp-W-sum}
        e^{\theta W(\widetilde{X})} &= \sum_{r=0}^\infty \dfrac{\theta^r W(\widetilde{X})^r}{r!}.
    \end{align}
    To compute the expectation, we need to verify that the conditions of the dominated convergence theorem are satisfied. To do so, we show that $\E{e^{|\theta| W(\widetilde{X})}}<\infty$. We have
    \begin{align*}
        \E{e^{|\theta| W(\widetilde{X})}} &= \sum_{r=0}^\infty \dfrac{|\theta|^r \E{W(\widetilde{X})^r}}{r!} \leq \sum_{r=0}^\infty \dfrac{|\theta|^r C_r}{r!},
    \end{align*}
    where $C_r$ is defined in \cref{lem:weina}. We now bound the last sum as follows:
    \begin{align*}
        \sum_{r=0}^\infty \dfrac{|\theta|^r C_r}{r!} &= 
        \sum_{r=0}^\infty \dfrac{(2B|\theta|)^r}{r!} + \sum_{r=0}^\infty \left((4 |\theta| \nu_{\max}) \left(\dfrac{G_{\max}\nu_{\max}+\gamma}{\gamma}\right)\right)^r \\
        &\stackrel{(*)}{=} e^{2B|\theta|} + \dfrac{\gamma}{\gamma - (4|\theta|\nu_{\max}) (G_{\max}\nu_{\max}+\gamma)},
    \end{align*}
    where $(*)$ holds by noticing that the first term is the Taylor expansion of the exponential function, and the second term corresponds to a geometric sum. The latter converges to the corresponding fraction if and only if
    \begin{align*}
        (4 |\theta| \nu_{\max}) \left(\dfrac{G_{\max}\nu_{\max}+\gamma}{\gamma}\right) < 1 \quad \iff\quad 
        |\theta| < \dfrac{\gamma}{4\nu_{\max}(G_{\max}\nu_{\max} + \gamma)}=\Theta.
    \end{align*}
    We have shown that
    \begin{align*}
        \E{e^{|\theta| W(\widetilde{X})}} \leq \sum_{r=0}^\infty \dfrac{|\theta|^r C_r}{r!} < \infty.
    \end{align*}
    Hence, we can interchange the sum and expectation in \eqref{eq:exp-W-sum}. Following similar steps to obtain a bound, we complete the proof that if $|\theta| < \Theta$, then
    \begin{align*}
        \E{e^{\theta W(\widetilde{X})}} &= \sum_{r=0}^\infty \dfrac{\theta^r \E{W(\widetilde{X})^r}}{r!} \\
        &\leq e^{2B\theta} + \dfrac{\gamma}{\gamma - (4\theta\nu_{\max}) (G_{\max}\nu_{\max}+\gamma)} \\
        &\leq e^{2B\Theta} + \dfrac{\gamma}{\gamma - (4\Theta\nu_{\max}) (G_{\max}\nu_{\max}+\gamma)}. \Halmos
    \end{align*}
\end{proof}

\section{Details of the JSQ system proof}
\label{sec:details-JSQ-thm-proof}

In this section we provide the details of the bound for \eqref{eq:mgf-ssc-step1}. We use Cauchy-Schwarz inequality, as follows:
\begin{align}
    \E{\sum_{j=1}^n \mu_{ij} \ind{\tq_j=0} (e^{s\epsilon n \tq_{\perp j}}-1)} 
    &\stackrel{(a)}{\leq} \sqrt{ \E{\sum_{j=1}^n \mu_{ij}^2 \ind{\tq_j=0} }} \sqrt{ \E{\sum_{j=1}^n (e^{ s\epsilon n \tq_{\perp j}}-1)^2} } \nonumber \\
    &\stackrel{(b)}{\leq} \sqrt{\mu_{\max}} \sqrt{\E{\sum_{j=1}^n \mu_{ij}\ind{\tq_j=0}}}  \sqrt{\E{\sum_{j=1}^n (e^{ s\epsilon n \tq_{\perp j}}-1)^2} } \nonumber \\
    &\stackrel{(c)}{=} \sqrt{\mu_{\max} \mu_\Sigma \epsilon} \sqrt{\E{\sum_{j=1}^n (e^{ s\epsilon n \tq_{\perp j}}-1)^2} },
    \label{eq:mgf-ssc-step2}
\end{align}
where 
$(a)$ uses the Cauchy-Schwarz inequality and that $\ind{\tq_j=0}^2=\ind{\tq_j=0}$; 
in $(b)$ we bounded $\mu_{ij} \leq \mu_{\max}$; 
and $(c)$ holds by \eqref{jsq-eq:service-prob-empty-queue}. 
Finally, we bound the last term as follows:
\begin{align}
    0\leq \left|\E{\sum_{j=1}^n (e^{ s\epsilon n \tq_{\perp j}}-1)^2} \right|
    &\stackrel{(a)}{\leq} \E{\sum_{j=1}^n |e^{ s\epsilon n \tq_{\perp j}}-1|^2}  \nonumber \\
    &\stackrel{(b)}{\leq} \E{\sum_{j=1}^n |s\epsilon n \tq_{\perp j}|^2 \,e^{2s\epsilon n|\tq_{\perp j}|} } \nonumber \\
    &\stackrel{(c)}{\leq} (s\epsilon n)^2 \sqrt{ \E{\sum_{j=1}^n |\tq_{\perp j}|^4} } \sqrt{ \E{\sum_{j=1}^n e^{4 s\epsilon n |\tq_{\perp j}|}} } ,
    \label{eq:mgf-ssc-step3}
  \end{align}
where $(a)$ holds by the triangle inequality; $(b)$ holds because $e^{|x|}-1 \leq |x| e^{|x|}$ for all $x\in\bR$; $(c)$ holds by the Cauchy-Schwarz inequality. To finalize the proof, we use the bounds from \cref{prop:ssc}. Specifically, for the first term, we have
\begin{align}\label{eq:mgf-ssc-step4}
    \E{\sum_{j=1}^n |\tq_{\perp j}|^4} &= \E{\|\tvq_\perp\|_4^4} 
    \stackrel{(a)}{\leq} \E{\|\tvq_\perp\|_{2}^4} 
    \stackrel{(b)}{\leq} C_4.
\end{align}
where 
$(a)$ holds by the inequality between the $L^2$ and $L^4$ norms; 
and $(b)$ by \eqref{eq:qperp-bound-poly} in \Cref{prop:ssc}. For the second term, notice
\begin{align}\label{eq:mgf-ssc-step5}
    \E{\sum_{j=1}^n e^{4s\epsilon n |\tq_{\perp j}|}} 
    &\leq n \E{e^{4s\epsilon n \|\tvq_{\perp}\|}}
    \stackrel{(*)}{\leq} n C_{\exp},    
\end{align}
where $(*)$ holds if and only if $4s\epsilon n < \Theta$. From the proof of \cref{prop:ssc}, we know
\begin{align*}
    \Theta &= \frac{\gamma}{4\nu_{\max}(G_{\max}\nu_{\max}+\gamma)} \\
    &= \frac{(1/2) \min_{i\in\cZ} \lambda_{i,\min}'}{4(1 + 1/\sqrt{n}) \left( (\lambda_{\max}n\mu_{\max}\max_{i\in\cZ} \alpha_{i\bullet})(1 + 1/\sqrt{n}) + (1/2)\min_{i\in\cZ} \lambda_{i,\min}'\right)}.
\end{align*}
By definition, $\epsilon < 1$. Hence, the proof holds for all $s < \frac{\Theta}{4n}$. Now we put everything together. Using \eqref{eq:mgf-ssc-step4} and \eqref{eq:mgf-ssc-step5} in \eqref{eq:mgf-ssc-step3}, we obtain
\begin{align} \label{eq:mgf-ssc-step3-ready}
    \left|\E{\sum_{j=1}^n (e^{s\epsilon n \tq_{\perp j}}-1)^2} \right| \leq (s\epsilon)^2 n^\frac{5}{2} \sqrt{C_4 C_{\exp}}.
\end{align}
Now using \eqref{eq:mgf-ssc-step3-ready} in \eqref{eq:mgf-ssc-step2}, we obtain
\begin{align}
    \E{\sum_{j=1}^n \mu_{ij} \ind{\tq_j=0} (e^{s\epsilon n \tq_{\perp j}}-1)} &\leq \sqrt{\mu_{\max}\mu_\Sigma \epsilon} \sqrt{(s\epsilon)^2 n^\frac{5}{2} \sqrt{C_4 C_{\exp}}} \nonumber \\
    &= \epsilon^{\frac{3}{2}} s n^{\frac{5}{4}} (\mu_{\max}\mu_{\Sigma})^{\frac{1}{2}}  (C_4C_{\exp})^{\frac{1}{4}}.
    \label{eq:mgf-ssc-step2-ready}
\end{align}
Then, plugging \eqref{eq:mgf-ssc-step2-ready} in \eqref{eq:mgf-ssc-step1} we obtain
\begin{align*}
    \E{e^{-s\epsilon \tq_\Sigma} \sum_{j=1}^n \mu_{ij}\ind{\tq_j=0} } \leq \mu_\Sigma \epsilon + \epsilon^{\frac{3}{2}} s n^{\frac{5}{4}} (\mu_{\max}\mu_{\Sigma})^{\frac{1}{2}}  (C_4C_{\exp})^{\frac{1}{4}} = \mu_\Sigma\epsilon + O(\epsilon^{\frac{3}{2}}).
\end{align*}

\end{APPENDICES}

% Acknowledgments here
%\ACKNOWLEDGMENT{}

% References here (outcomment the appropriate case)

% CASE 1: BiBTeX used to constantly update the references
%   (while the paper is being written).

\newpage 

\bibliographystyle{informs2014} % outcomment this and next line in
\bibliography{references} % if more than one, comma separated

@book{Asm-2003-Applied-Prob-book,
  title={Applied probability and queues},
  author={Asmussen, S{\o}ren},
  year={2003},
  publisher={Springer}
}

@article{Bac-Mak-1991-M-G-1-martingale,
  title={Martingale relations for the {M/GI/1} queue with Markov modulated Poisson input},
  author={Baccelli, Fran{\c{c}}ois and Makowski, Armand M},
  journal={Stochastic processes and their applications},
  volume={38},
  number={1},
  pages={99--133},
  year={1991},
  publisher={Elsevier}
}

@article{Ban-Muk-2019-JSQ-HW,
author = {Sayan Banerjee and Debankur Mukherjee},
title = {{Join-the-shortest queue diffusion limit in Halfin–Whitt regime: Tail asymptotics and scaling of extrema}},
volume = {29},
journal = {The Annals of Applied Probability},
number = {2},
publisher = {Institute of Mathematical Statistics},
pages = {1262 -- 1309},
year = {2019},
doi = {10.1214/18-AAP1436},
URL = {https://doi.org/10.1214/18-AAP1436}
}

@article{Ban-Muk-2020-JSQ-HW-sensitivity,
 ISSN = {10505164, 21688737},
 _URL = {https://www.jstor.org/stable/26907045},
 author = {Sayan Banerjee and Debankur Mukherjee},
 journal = {The Annals of Applied Probability},
 number = {1},
 pages = {pp. 80--144},
 publisher = {Institute of Mathematical Statistics},
 title = {JOIN-THE-SHORTEST QUEUE DIFFUSION LIMIT IN HALFIN–WHITT REGIME: SENSITIVITY ON THE HEAVY-TRAFFIC PARAMETER},
 urldate = {2025-10-23},
 volume = {30},
 year = {2020}
}

@article{Bell-Will-01-Nsystem,
  title={Dynamic scheduling of a system with two parallel servers in heavy traffic with resource pooling: Asymptotic optimality of a threshold policy},
  author={Bell, Steven L and Williams, Ruth J},
  journal={Annals of Applied Probability},
  volume={11},
  number={3},
  pages={608--649},
  year={2001},
  publisher={Institute of Mathematical Statistics}
}

@article{Blo-et-al-2014-Markov-G-infty,
  title={Markov-modulated infinite-server queues with general service times},
  author={Blom, J and Kella, Offer and Mandjes, Michel and Thorsdottir, Halldora},
  journal={Queueing Systems},
  volume={76},
  number={4},
  pages={403--424},
  year={2014},
  publisher={Springer}
}

@article{Bra-2020-JSQ,
  title={Steady-state analysis of the join-the-shortest-queue model in the {Halfin--Whitt} regime},
  author={Braverman, Anton},
  journal={Mathematics of Operations Research},
  volume={45},
  number={3},
  pages={1069--1103},
  year={2020},
  publisher={INFORMS}
}

@article{Bra-Dai-2017-Stein,
  title={Stein's method for steady-state diffusion approximations: {A}n introduction through the {E}rlang-{A} and {E}rlang-{C} models},
  author={Braverman, Anton and Dai, JG and Feng, Jiekun},
  journal={Stochastic Systems},
  volume={6},
  number={2},
  pages={301--366},
  year={2017},
  publisher={INFORMS}
}

@article{Bra-Dai-Miy_2017_BAR,
  title={Heavy traffic approximation for the stationary distribution of a {G}eneralized {J}ackson {N}etwork: The {BAR} approach},
  author={Braverman, Anton and Dai, JG and Miyazawa, Masakiyo},
  journal={Stochastic Systems},
  volume={7},
  number={1},
  pages={143--196},
  year={2017},
  publisher={INFORMS}
}

@article{Bra-Dai-2017-Stein2,
	title={Stein's Method for Steady-State Diffusion Approximations of {M/Ph/n+ M} Systems},
	author={Braverman, A. and Dai, J.G.},
	journal={The Annals of Applied Probability},
	volume={27},
	number={1},
	pages={550--581},
	year={2017},
	publisher={Institute of Mathematical Statistics}
}

@article{Bra-Dai-Miy-2025-BAR,
  title={The {BAR} approach for multiclass queueing networks with {SBP} service policies},
  author={Braverman, Anton and Dai, JG and Miyazawa, Masakiyo},
  journal={Stochastic systems},
  volume={15},
  number={1},
  pages={1--49},
  year={2025},
  publisher={INFORMS}
}

@article{Bur-Smi_1986_Markov-G-1-HT,
  title={An asymptotic analysis of a queueing system with {Markov}-modulated arrivals},
  author={Burman, David Y and Smith, Donald R},
  journal={Operations Research},
  volume={34},
  number={1},
  pages={105--119},
  year={1986},
  publisher={INFORMS}
}

@article{Che-Mou-Mag-2022-SGD-Transform,
  title={Stationary behavior of constant stepsize {SGD} type algorithms: An asymptotic characterization},
  author={Chen, Zaiwei and Mou, Shancong and Maguluri, Siva Theja},
  journal={Proceedings of the ACM on Measurement and Analysis of Computing Systems},
  volume={6},
  number={1},
  pages={1--24},
  year={2022},
  publisher={ACM New York, NY, USA}
}

@article{Dai-Huo-2024-Multi-Scale-HT,
  title={Asymptotic Product-form Steady-state for Multiclass Queueing Networks: A Reentrant Line Case Study},
  author={Dai, Jim and Huo, Dongyan},
  journal={arXiv preprint arXiv:2411.00930},
  year={2024}
}

@article{Dim-2011-Falin-Generalization,
  title={Single-server queueing system with {Markov}-modulated arrivals and service times},
  author={Dimitrov, Mitko},
  journal={Pliska Stud. Math. Bulg},
  volume={20},
  pages={53--62},
  year={2011}
}

@article{Dai-Gly-Yao-2023-Multi-Scale-HT,
  title={Asymptotic product-form steady-state for generalized {Jackson} networks in multi-scale heavy traffic},
  author={Dai, JG and Glynn, Peter and Xu, Yaosheng},
  journal={arXiv preprint arXiv:2304.01499},
  year={2023}
}

@article{Eph-1980-JSQ,
  title={A simple dynamic routing problem},
  author={Ephremides, Anthony and Varaiya, Pravin and Walrand, Jean},
  journal={IEEE transactions on Automatic Control},
  volume={25},
  number={4},
  pages={690--693},
  year={1980},
  publisher={IEEE}
}

@article{Ery-Sri-2012-drift,
  title={Asymptotically tight steady-state queue length bounds implied by drift conditions},
  author={Eryilmaz, Atilla and Srikant, R},
  journal={Queueing Systems},
  volume={72},
  number={3},
  pages={311--359},
  year={2012},
  publisher={Springer}
}

@article{Esc-Gam-2018-JSQ,
  title={Join the shortest queue with many servers. {The} heavy-traffic asymptotics},
  author={Eschenfeldt, Patrick and Gamarnik, David},
  journal={Mathematics of Operations Research},
  volume={43},
  number={3},
  pages={867--886},
  year={2018},
  publisher={INFORMS}
}

@article{Fal-Fal-1999-Markov-G-1-HT,
  title={Heavy traffic analysis of {M/G/1} type queueing systems with {Markov}-modulated arrivals},
  author={Falin, Gennadi and Falin, Anatoli},
  journal={Top},
  volume={7},
  number={2},
  pages={279--291},
  year={1999},
  publisher={Springer}
}

@article{Fie-2025-working-vacation-survey,
    AUTHOR = {Fiems, Dieter},
    TITLE = {Queues with Working Vacations: A Survey},
    JOURNAL = {Mathematics},
    VOLUME = {13},
    YEAR = {2025},
    NUMBER = {11},
    ARTICLE-NUMBER = {1894},
    URL = {https://www.mdpi.com/2227-7390/13/11/1894},
    ISSN = {2227-7390},
    DOI = {10.3390/math13111894}
}

@article{Fos-Sal-1978-jsq,
  title={A basic dynamic routing problem and diffusion},
  author={Foschini, G and Salz, JACK},
  journal={IEEE Transactions on Communications},
  volume={26},
  number={3},
  pages={320--327},
  year={1978},
  publisher={IEEE}
}

@article{Gro-Scu-Har-2019-guardrails,
  title={Load balancing guardrails: Keeping your heavy traffic on the road to low response times},
  author={Grosof, Isaac and Scully, Ziv and Harchol-Balter, Mor},
  journal={Proceedings of the ACM on Measurement and Analysis of Computing Systems},
  volume={3},
  number={2},
  pages={1--31},
  year={2019},
  publisher={ACM New York, NY, USA}
}

@incollection{Har-1988-Brownian,
  title={Brownian Models of Queueing Networks with Heterogeneous Customer Populations},
  author={Harrison, J},
  booktitle={Stochastic Differential Systems, Stochastic Control Theory and Applications},
  pages={147--186},
  year={1988},
  publisher={Springer}
}

@book{Har_2013_book,
    place={Cambridge},
    title={Brownian Models of Performance and Control},
    DOI={10.1017/CBO9781139087698},
    publisher={Cambridge University Press},
    author={Harrison, J.}, year={2013}
}

@article{HarLop-1999-CRP,
title = {Heavy traffic resource pooling in parallel-server systems},
author = {Harrison, J and L{\'o}pez, M},
journal = {Queueing Systems},
pages = {339-368},
year = {1999},
}

@article{HL-Mag-2020-MGF,
  title={Transform methods for heavy-traffic analysis},
  author={Hurtado-Lange, Daniela and Maguluri, Siva Theja},
  journal={Stochastic Systems},
  volume={10},
  number={4},
  pages={275--309},
  year={2020},
  publisher={INFORMS}
}

@article{HL-Mag-2022-JSQ-Many-Server,
  title={A load balancing system in the many-server heavy-traffic asymptotics},
  author={Hurtado-Lange, Daniela and Maguluri, Siva Theja},
  journal={Queueing Systems},
  volume={101},
  number={3},
  pages={353--391},
  year={2022},
  publisher={Springer}
}

@article{Jhu-HL-Mag-2023-JSQ-ROC,
  title={Exponential Tail Bounds on Queues},
  author={Jhunjhunwala, Prakirt and Hurtado-Lange, Daniela and Theja Maguluri, Siva},
  journal={ACM SIGMETRICS Performance Evaluation Review},
  volume={51},
  number={2},
  pages={24--26},
  year={2023},
  publisher={ACM New York, NY, USA}
}

@article{Jhu-Zub-Mag_2025_JSQ-abandonments,
  title={Join-the-shortest queue with abandonment: Critically loaded and heavily overloaded regimes},
  author={Jhunjhunwala, Prakirt R and Zubeldia, Martin and Maguluri, Siva Theja},
  journal={Mathematics of Operations Research},
  year={2025},
  publisher={INFORMS}
}

@article{Kin_1962_RBM,
  title={On queues in heavy traffic},
  author={Kingman, J},
  journal={Journal of the Royal Statistical Society. Series B (Methodological)},
  pages={383--392},
  year={1962},
  publisher={JSTOR}
}

@article{Liu-Yin-2020-Stein-SubHW,
  title={Steady-state analysis of load-balancing algorithms in the sub-{Halfin--Whitt} regime},
  author={Liu, Xin and Ying, Lei},
  journal={Journal of Applied Probability},
  volume={57},
  number={2},
  pages={578--596},
  year={2020},
  publisher={Cambridge University Press}
}

@article{Luo-Zub-2025-Load-Balancing-Stability,
  title={Stability and Heavy-traffic Delay Optimality of General Load Balancing Policies in Heterogeneous Service Systems},
  author={Luo, Yishun and Zubeldia, Martin},
  journal={arXiv preprint arXiv:2510.14284},
  year={2025}
}

@article{Mag-Sri-2016-switch,
  title={Heavy traffic queue length behavior in a switch under the {MaxWeight} algorithm},
  author={Maguluri, Siva Theja and Srikant, R},
  journal={Stochastic Systems},
  volume={6},
  number={1},
  pages={211--250},
  year={2016},
  publisher={INFORMS}
}

@inproceedings{Mit-1996-JSQ-d,
  title={Load balancing and density dependent jump {Markov} processes},
  author={Mitzenmacher, Michael},
  booktitle={Proceedings of 37th Conference on Foundations of Computer Science},
  pages={213--222},
  year={1996},
  organization={IEEE}
}

@article{Mou-Mag-2024-Switch-Markov,
  title={Heavy-traffic queue length behavior in a switch under {Markovian} arrivals},
  author={Mou, Shancong and Maguluri, Siva Theja},
  journal={Advances in Applied Probability},
  volume={56},
  number={3},
  pages={1106--1152},
  year={2024},
  publisher={Cambridge University Press}
}

@article{Muk-2022-ROC-JSQ,
  title={Rates of convergence of the join the shortest queue policy for large-system heavy traffic},
  author={Mukherjee, Debankur},
  journal={Queueing Systems},
  volume={100},
  number={3},
  pages={317--319},
  year={2022},
  publisher={Springer}
}

@article{Pra-Zhu-1989-Markov-q,
  title={Markov-modulated queueing systems},
  author={Prabhu, NU and Zhu, Yixin},
  journal={Queueing systems},
  volume={5},
  number={1},
  pages={215--245},
  year={1989},
  publisher={Springer}
}

@article{Var-Mag-2022-HT-two-sided-q,
  title={A heavy traffic theory of two-sided queues},
  author={Mahavir Varma, Sushil and Theja Maguluri, Siva},
  journal={ACM SIGMETRICS Performance Evaluation Review},
  volume={49},
  number={3},
  pages={43--44},
  year={2022},
  publisher={ACM New York, NY, USA}
}

@article{Vve-Dob-Kar-1996-JSQ2,
  title={Queueing system with selection of the shortest of two queues: An asymptotic approach},
  author={Vvedenskaya, Nikita Dmitrievna and Dobrushin, Roland L'vovich and Karpelevich, Fridrikh Izrailevich},
  journal={Problemy Peredachi Informatsii},
  volume={32},
  number={1},
  pages={20--34},
  year={1996},
  publisher={Russian Academy of Sciences, Branch of Informatics, Computer Equipment and~…}
}

@article{Wal_2020_SteinHT,
  title={Stein's Method for the Single Server Queue in Heavy Traffic},
  author={Gaunt, Robert and Walton, Neil},
  journal={Statistics \& Probability Letters},
  volume={156},
  pages={108566},
  year={2020},
  publisher={Elsevier}
}

@article{Wang-Mag-Sri-2022-Bandwidth-Sharing,
  title={Heavy-traffic insensitive bounds for weighted proportionally fair bandwidth sharing policies},
  author={Wang, Weina and Maguluri, Siva Theja and Srikant, R and Ying, Lei},
  journal={Mathematics of Operations Research},
  volume={47},
  number={4},
  pages={2691--2720},
  year={2022},
  publisher={INFORMS}
}

@article{Web-1978-JSQ-optimal,
  title={On the optimal assignment of customers to parallel servers},
  author={Weber, Richard R},
  journal={Journal of Applied Probability},
  volume={15},
  number={2},
  pages={406--413},
  year={1978},
  publisher={Cambridge University Press}
}

@article{Williams-2000-CRP,
  title={On dynamic scheduling of a parallel server system with complete resource pooling},
  author={Williams, R},
  journal={Fields Institute Communications},
  volume={28},
  number={49-71},
  pages={5--1},
  year={2000},
  publisher={Citeseer}
}

@article{Win-1977-JSQ-optimality,
  title={Optimality of the shortest line discipline},
  author={Winston, Wayne},
  journal={Journal of applied probability},
  volume={14},
  number={1},
  pages={181--189},
  year={1977},
  publisher={Cambridge University Press}
}

@article{Ying-2017-Stein,
  title={Stein's method for mean field approximations in light and heavy traffic regimes},
  author={Ying, Lei},
  journal={Proceedings of the ACM on Measurement and Analysis of Computing Systems},
  volume={1},
  number={1},
  pages={1--27},
  year={2017},
  publisher={ACM New York, NY, USA}
}

@article{Gup-2007-Analysis,
title = {Analysis of join-the-shortest-queue routing for web server farms},
journal = {Performance Evaluation},
volume = {64},
number = {9},
pages = {1062-1081},
year = {2007},
note = {Performance 2007},
issn = {0166-5316},
doi = {https://doi.org/10.1016/j.peva.2007.06.012},
url = {https://www.sciencedirect.com/science/article/pii/S0166531607000624},
author = {Varun Gupta and Mor {Harchol Balter} and Karl Sigman and Ward Whitt},
}

@article{Bambos-2004-Queueing, title={Queueing and scheduling in random environments}, volume={36}, DOI={10.1239/aap/1077134474}, number={1}, journal={Advances in Applied Probability}, author={Bambos, Nicholas and Michailidis, George}, year={2004}, pages={293–317}}

@article{Whitt-2018-Varying,
  title={Time-Varying Queues},
  author={Whitt, Ward},
  journal={Queueing Models and Service Management},
  volume={1},
  number={2},
  pages={79--164},
  year={2018}
}

%%%%%%%%%%%%%%%%%
\end{document}